\documentclass[final,3p,times,sort & compress]{elsarticle}

\usepackage{lineno}
\usepackage{graphicx,amsmath,bm,color,geometry}
\usepackage{algorithm}
\usepackage{algorithmicx, algpseudocode}
\usepackage[colorlinks,
            linkcolor=blue,
            anchorcolor=blue,
            citecolor=blue]{hyperref}
\usepackage{booktabs}
\usepackage{caption}
\usepackage{subcaption}
\usepackage{threeparttable}
\usepackage{multirow}
\usepackage{amssymb,amsthm}

\newtheorem{theorem}{Theorem}
\newtheorem{remark}{Remark}
\newtheorem{lemma}{Lemma}

\newcommand{\ie}{\textit{i}.\textit{e}.}

\newcommand{\Var}{\mathrm{var}}
\newcommand{\Cov}{\mathrm{Cov}}
\newcommand{\SE}{\text{SE}}
\newcommand{\SSmy}{\text{SS}}
\newcommand{\SD}{\text{SD}}
\newcommand{\DS}{\text{DS}}
\newcommand{\DD}{\text{DD}}
\newcommand{\NR}{\text{NR}}
\newcommand{\MF}{\text{mf}}
\newcommand{\curly}[1]{\left\{ #1 \right\}}
\DeclareMathOperator{\diag}{diag}
\journal{ArXiv}

\begin{document}

\begin{frontmatter}

\title{Sparse discovery of differential equations based on multi-fidelity Gaussian process \tnoteref{label-title}}
\tnotetext[label-title]{This work is partially supported by the National Natural Science Foundation of China (NSFC) under grant number 12101407, the Chongqing Entrepreneurship and Innovation Program for Returned Overseas Scholars under grant number CX2023068, and the Fundamental Research Funds for the Central Universities under grant number 2023CDJXY-042.}

\author[label-addr1]{Yuhuang Meng}
\ead{mengyh@shanghaitech.edu.cn}
\author[label-addr2,label-addr3]{Yue Qiu \corref{label-cor}}
\ead{qiuyue@cqu.edu.cn}
\cortext[label-cor]{Corresponding author.}

\affiliation[label-addr1]{
	organization={School of Information Science and Technology, ShanghaiTech University},
	city={Shanghai},
	postcode={201210},
	country={China.}}
\affiliation[label-addr2]{
	organization={College of Mathematics and Statistics, Chongqing University},
	city={Chongqing},
	postcode={401331},
	country={China.}}
\affiliation[label-addr3]{
	organization={Key Laboratory of Nonlinear Analysis and its Applications (Chongqing University), Ministry of Education},
	city={Chongqing},
	postcode={401331},
	country={China.}}

\begin{abstract}
	Sparse identification of differential equations aims to compute the analytic expressions from the observed data explicitly. However, there exist two primary challenges. Firstly, it exhibits sensitivity to the noise in the observed data, particularly for the derivatives computations. Secondly, existing literature predominantly concentrates on single-fidelity (SF) data, which imposes limitations on its applicability due to the computational cost. In this paper, we present two novel approaches to address these problems from the view of uncertainty quantification. We construct a surrogate model employing the Gaussian process regression (GPR) to mitigate the effect of noise in the observed data, quantify its uncertainty, and ultimately recover the equations accurately. Subsequently, we exploit the multi-fidelity Gaussian processes (MFGP) to address scenarios involving multi-fidelity (MF), sparse, and noisy observed data. We demonstrate the robustness and effectiveness of our methodologies through several numerical experiments.
\end{abstract}




\end{frontmatter}


\section{Introduction}
\label{sec:Introduction}

Nonlinear differential equations are widely prevalent in both science and engineering applications. Nonetheless, these equations are generally unknown in many situations, which makes it difficult to understand and control the systems of interest. Fortunately, amounts of data concerning the systems of interest can be available through experiments or simulations. To address this challenge, data-driven approaches have emerged. For instance, the Koopman operator theory and dynamic mode decomposition (DMD) embed a nonlinear system into a higher-dimensional linear space via a set of observation functions, which provides a powerful and efficient tool for understanding the behavior of the nonlinear dynamic systems, especially in fluid dynamics \cite{koopman1931hamiltonian, kutz2016dynamic, lusch2018deep}.

In recent years, numerous efforts have been dedicated to learn these nonlinear equations from the observed data. Based on the diversity of the final results, these methods can be broadly categorized into two classes. The first develops black-box models to approximate the underlying differential equations. Such techniques rely on deep neural networks and numerical schemes. \citet{raissi2018multistep} combined the multistep method with deep neural networks to discover nonlinear ordinary differential equations (ODEs). They represented the right-hand-side of ODEs using deep neural networks. In a similar vein, \citet{raissi2018deep} harnessed deep neural networks and employed automatic differentiation to compute partial derivatives. \citet{rudy2019deep} introduced a robust identification algorithm for dynamic systems that combines neural networks with the Runge-Kutta method to denoise the data and discover the dynamic system simultaneously. \citet{qin2019data} used a residual network to approximate the integral form of the underlying ODEs. Conversely, the second approach seeks an analytic expression for the hidden differential equations. Sparse Identification of Nonlinear Dynamic Systems (SINDy), introduced by \citet{brunton2016discovering}, has emerged as a significant framework. It assumes the underlying differential equations could be represented by few predefined functions. The core procedure of SINDy involves the computation of derivatives and the solution of linear equations while imposing constraints for a sparse solution. \citet{long2019pdenet} employed convolutional neural networks to learn the symbolic representations of partial differential equations (PDEs). \citet{kim2021integration} pursued symbolic regression using neural networks. \citet{kang2021ident} proposed IDENT to obtain sparse solutions through Lasso and validated these solutions by assessing time evolution errors.

SINDy has proven to be a successful tool in the discovery of PDEs, primarily due to its simplicity and efficiency \cite{rudy2017datadriven}. However, there remain two critical challenges. The first involves robust approximation of the temporal or spatial partial derivatives when the observed data is corrupted by noise. In practice, only the states observed data is available whereas the temporal derivatives or partial derivatives are required to be calculated using numerical differentiation methods. However, these numerical differentiation methods magnify the noise in the observed data, which makes SINDy sensitive to noise. The second revolves around effectively leveraging the multi-fidelity data to reduce the computational cost.

Regarding the first challenge, various robust algorithms have emerged. The first category constructs a local polynomial surrogate model aimed to smooth the noisy observed data and obtain the derivatives analytically. For instance, \citet{he2022robust} employed the moving least squares (MLS) to smooth the observed data and enhance the stability of numerical differentiation when applied to noisy datasets. Similarly, \citet{sun2023pisl} represented the observed data by a set of cubic spline basis functions and the surrogate model and discovered equations are trained simultaneously. Notably, the noisy data is fitted by local quadratic and cubic polynomials in \cite{he2022robust} and \cite{sun2023pisl}, respectively. Analogous methodologies have appeared in \cite{cortiella2022priori, sandoz2023sindy, vanbreugel2020numerical}. While these approaches show promising results in specific scenarios, they primarily rely on the local regression models that entail manual selection of parameters (such as the window length) and require analysis case-by-case. The second category seeks to recover nonlinear equations from the perspective of sparse Bayesian learning. Relevance vector machine (RVM) was introduced to infer the nonlinear equations \cite{zhang2018robust, fuentes2021equation}. Furthermore, Tikhonov regularization is utilized for derivatives computations \cite{rudy2017datadriven, cortiella2021sparse}, and neural networks were employed as a surrogate model to smooth the data, which transforms numerical differentiation to automatic differentiation to decrease the impact of noise \cite{lagergren2020learning, robertstephany2023pdelearn}. Moreover, in addition to the aforementioned approaches to smooth noisy data, there are several derivative-free methods. Notably, RK4-SINDy eliminates the approximations of derivatives by employing the fourth-order Runge-Kutta method \cite{goyal2022discovery}. On the other hand, weak SINDy (WSINDy) avoids the derivative approximations by constructing the weak form of the underlying differential equations \cite{messenger2021weak, messenger2021weaka}.

As for the second challenge, related work mainly concentrate on the single-fidelity, namely high-fidelity (HF) data. Nevertheless, the accurate simulation of HF data suffers from the large computational cost. To address this issue, the integration of a small amount of HF data with suitable low-fidelity (LF) observed data which may be of lower accuracy but also lower cost, becomes a feasible approach. This is commonly referred to as the multi-fidelity modelling. The problem of harnessing multi-fidelity data for sparse identification remains unresolved. One classical MF modeling technique is the linear autoregresive method known as co-kriging \cite{kennedy2000predicting, legratiet2014recursive}, where each fidelity level model is represented by a Gaussian process and the relationship between the outputs of the LF and HF model is assumed to be linear. \citet{perdikaris2017nonlinear} proposed the nonlinear autoregressive multi-fidelity GP regression (MFGP) scheme which extends the capabilities of the linear autoregresive approach. Moreover, recent developments have seen the emergence of deep neural networks for multi-fidelity modeling \cite{meng2020composite, chakraborty2021transfer, conti2023multifidelitya}.

This paper aims to address the aforementioned challenges by introducing two methodologies. Firstly, we propose the Gaussian Process based Sparse Identification of Nonlinear Dynamics (GP-SINDy) to construct a global surrogate model, alleviating the effect of noise in the observed data. GP-SINDy leverages the non-parametric model GPR to effectively smooth the noisy data. Meanwhile, the derivatives are computed analytically within the GP framework. Notably, GPR provides valuable uncertainty quantification (UQ) information for the variables inference. We incorporate this UQ information into the weighted least squares (WLS) problem to ensure the accurate recovery of the potential functions. Furthermore, we introduce the Multi-Fidelity Gaussian Process based Sparse Identification of Nonlinear Dynamics (MFGP-SINDy) to infer the explicit representation of the differential equations using a suitable amount of LF data and limited HF data. To achieve this, MFGP \cite{perdikaris2017nonlinear} is utilized for the information fusion among different fidelity levels. The main contributions of this paper include:
\begin{enumerate}
    \item A robust algorithm for sparse identification of differential equation using GP, aka GP-SINDy, is proposed.
    \item The uncertainty of time derivative is approximated by its posterior variance in GP and this UQ information are embodied into the process of sparse identification of differential equation.
    \item MFGP-SINDy is developed to cope with the case of multi-fidelity, sparse, and noisy observed data. Meanwhile, the partial derivative computations of MFGP kernel are provided, which is the key for MFGP-SINDy.
\end{enumerate}

This paper is organized as follows. We present the problem of sparse discovery for differential equations in Section \ref{sec:Problem}. In Section \ref{sec:Methods}, we briefly review the Gaussian process regression and present our GP-SINDy and MFGP-SINDy algorithms. Several numerical experiments in Section \ref{sec:Numerical} are carried out to demonstrate the efficiency and robustness of our methods, and we summarize our paper in Section \ref{sec:Conclusion}.

\section{Problem statement}
\label{sec:Problem}

Consider a (nonlinear) differential equation (ODE or PDE) described by the following form
\begin{equation}\label{equ:nonlinear-1}
    \frac{d}{d t} \mathbf{u}=\mathbf{f}(\mathbf{u}),
\end{equation}
where $\mathbf{u} = [u_1(\mathbf{x}), u_2(\mathbf{x}), \cdots, u_d(\mathbf{x})]^T\in \mathbb{R}^{d}$ represents the state variables, the (nonlinear) evolution $\mathbf{f}(\mathbf{u})$ of the system is unknown, and $\mathbf{x}\in\mathbb{R}^D$ represent the time (and the space location), where $\mathbf{x}=t$ for ODEs, and $\mathbf{x} = (t, x)$ for PDEs, with $t \in [0, T]$. For PDEs, the right-hand-side of \eqref{equ:nonlinear-1} also contains the partial derivatives with respect to (w.r.t.) $x$, \ie, $\mathbf{f}(\mathbf{u}) = \mathbf{f}(\mathbf{u}, \mathbf{u}_x, \mathbf{u}_{xx}, \cdots)$.

The available observed data is represented by $\mathcal{D} =\{\mathbf{x}_{i} , \mathbf{u}_{i} \}$ for $i=1,\cdots ,N$, where $\mathbf{u}_{i} = \mathbf{u}_{i}^* + \boldsymbol{\varepsilon}_i$. Here, $\mathbf{u}_{i}^*$ denotes the clean data sampled from the true system, and $\boldsymbol{\varepsilon}_i$ is independent and identically distributed (i.i.d.) Gaussian white noise with $\boldsymbol{\varepsilon}_i \sim \mathcal{N}(\mathbf{0}_d, \sigma^2\mathbf{I}_d)$, where $\mathbf{0}_d$ and $\mathbf{I}_d$ denote the zero vector of size $d$ and the identity matrix of size $d \times d$, respectively. We could assemble the data $\mathcal{D}$ into the matrix form with $\mathbf{X} = [\mathbf{x}_{1}, \mathbf{x}_{2}, \cdots, \mathbf{x}_N] \in \mathbb{R}^{D\times N}$, $\mathbf{U} = [\mathbf{u}_{1}, \mathbf{u}_{2}, \cdots, \mathbf{u}_N]^T \in \mathbb{R}^{N\times d}$ and
\[
    \mathbf{U} = \mathbf{U}^* + \mathbf{E},
\]
where $\mathbf{U}^*$ represents the corresponding clean data matrix and $\mathbf{E}$ is the noise matrix. Our objective is to discover the explicit expression of $\mathbf{f}(\mathbf{u})$ based on the noisy observed data $\mathbf{U}$.

Assume that $\mathbf{f}(\mathbf{u})$ can be expressed by a linear combination of as few as predefined functions from a function library denoted by $\boldsymbol{\Phi} (\mathbf{u} )=[\phi_1(\mathbf{u} ), \phi_2(\mathbf{u} ), \cdots, \phi_{N_f}(\mathbf{u} )]$ where each $\phi_i$ is referred to as a basis function or function feature, and $N_f$ is the size of the function library. The function library $\boldsymbol{\Phi}(\mathbf{u})$ is often selected as the polynomial basis functions or trigonometric functions. For PDE systems, the basis functions also involve the partial derivatives w.r.t. the space variables, such as $\mathbf{u}_{x}$, $\mathbf{u}_{xx}$, and their combination functions $\mathbf{u}\mathbf{u}_{x}$, $\mathbf{u}\mathbf{u}_{xx}$, and $\mathbf{u}_x\mathbf{u}_{xx}$. The task of discovering the (nonlinear) system \eqref{equ:nonlinear-1} can be reformulated as the problem of solving the following linear equations while enforcing the constraint of sparse solution,
\begin{equation}\label{equ:nonlinear-2}
    \dot{\mathbf{U}} = \boldsymbol{\Phi}\mathbf{C} + \mathbf{E}'.
\end{equation}
Here, $\dot{\mathbf{U}} \in \mathbb{R}^{N\times d}$ represents the temporal derivatives of the state variables, which needs to be computed from the given data matrix $\mathbf{X}$ and $\mathbf{U}$. $\mathbf{E}'\in \mathbb{R}^{N\times d}$ is the noise matrix. The matrix $\boldsymbol{\Phi}\in\mathbb{R}^{N\times N_f}$ corresponds to the function library, and $\mathbf{C} \in \mathbb{R}^{N_f\times d}$ represents the sparse coefficient matrix to be determined.

The problem of sparse discovery of differential equations described by \eqref{equ:nonlinear-2} involves solving a linear least square problem. Partition $\mathbf{U}$, $\dot{\mathbf{U}}$ and $\mathbf{C}$ using $\mathbf{U} = [\mathbf{U}_1, \mathbf{U}_2, \cdots, \mathbf{U}_d]$, $\dot{\mathbf{U}} = [\dot{\mathbf{U}}_1, \dot{\mathbf{U}}_2, \cdots, \dot{\mathbf{U}}_d]$, and $\mathbf{C} = [\mathbf{C}_1, \mathbf{C}_2, \cdots, \mathbf{C}_d]$, where $\mathbf{U}_i$, $\dot{\mathbf{U}}_i \in\mathbb{R}^{N}$, and $\mathbf{C}_i \in\mathbb{R}^{N_f}$ represent the $i$-th column of matrices $\dot{\mathbf{U}}$ and $\mathbf{C}$, respectively. We aim to compute a sparse solution $\mathbf{C}_i\in\mathbb{R}^{N_f}$ with
\begin{equation}\label{equ:SINDy-opt-2}
    \min_{\mathbf{C}_i} \left(\dot{\mathbf{U}}_{i} -\boldsymbol{\Phi }\mathbf{C}_{i}\right)^{T}\mathbf{W}_{i}\left(\dot{\mathbf{U}}_{i} -\boldsymbol{\Phi }\mathbf{C}_{i}\right) + \lambda_0 \| \mathbf{C}_i \| _{0},
\end{equation}
for each $i=1, 2, \cdots, d$. We transform the problem of sparse discovery of differential equations into a WLS problem.

Here $\mathbf{W}_i$ is the weight matrix and should be positive semi-definite and the $\ell_0$ regularization term is employed to promote the sparsity of $\mathbf{C}_i$. If the weight matrix $\mathbf{W}_i$ is chosen as the identity matrix, this problem degenerates to an ordinary least squares formulation which is studied in classical SINDy \cite{brunton2016discovering, rudy2017datadriven}.

\section{Sparse identification using Gaussian process}
\label{sec:Methods}
In this section, we propose two methods that leverage the GPR and MFGP to recover the nonlinear differential equations, named GP-SINDy and MFGP-SINDy, respectively. First, we briefly review some fundamental concepts of GPR. Subsequently, we present the GP-SINDy approach, which incorporates the inference outcomes of GPR into the sparse discovery of nonlinear systems. Finally, we employ MFGP to fuse the nonlinear information among different fidelity levels for the sparse identification.

\subsection{Gaussian process regression}
Assume that the observation $\mathcal{D} = \curly{(\mathbf{x}_{i} ,y_{i})\Big | ,i=1,\cdots ,N}$ satisfies the following model,
\begin{equation*}
    y_i = f(\mathbf{x}_i) + \varepsilon_i.
\end{equation*}
Here, $\mathbf{x}_{i} \in \Omega \subset \mathbb{R}^D$, and the scalar $y_{i}$ is contaminated by i.i.d. noise $\varepsilon_i \sim \mathcal{N}(0, \sigma_{0}^2)$. The function $f$ could be characterized by a Gaussian process, \ie, $f(\mathbf{x}) \sim \mathcal{GP}\left(\overline{f(\mathbf{x})} ,k\left( \mathbf{x},\mathbf{x}^{\prime } ;\boldsymbol{\theta}\right)\right)$, where $\boldsymbol{\theta}$ is the hyperparameters to be determined, the mean function $\overline{f(\mathbf{x})}$ and the covariance function (kernel function) $k\left( \mathbf{x},\mathbf{x}^{\prime } ;\boldsymbol{\theta}\right)$ are defined by,
\[
 \begin{cases}
    \overline{f(\mathbf{x})}=\mathbb{E} (f(\mathbf{x})),\\
    k\left( \mathbf{x},\mathbf{x}^{\prime } ;\boldsymbol{\theta}\right) =\mathbb{E}\left[ \left(f(\mathbf{x})-\overline{f(\mathbf{x})}\right)^{T}\left(f\left( \mathbf{x}^{\prime }\right) -\overline{f\left( \mathbf{x}^{\prime }\right)}\right)\right].
\end{cases}
\]
In this paper, we choose the squared exponential (SE) kernel, \ie, $k_\SE\left(\mathbf{x} ,\mathbf{x}^{\prime } ;\boldsymbol{\theta}\right) =\boldsymbol{\theta} _{0}^{2}\exp\left( -\sum \limits _{s=1}^{D}\frac{\left(\mathbf{x}_{s} -\mathbf{x}_{s}^{\prime }\right)^{2}}{2\boldsymbol{\theta} _{s}^{2}}\right)$. For simplicity, we denote the training data by $\mathbf{X} =[\mathbf{x}_{1}, \mathbf{x}_{2}, \cdots, \mathbf{x}_N] \in \mathbb{R}^{D\times N}$ and $\mathbf{y} = [y_{1}, y_{2}, \cdots, y_N]^T\in \mathbb{R}^{N}$.

\subsubsection{Inferring the state variables}
GPR provides a non-parameter model that allows one to infer the quantities of interest. Given the training data $\mathbf{X}$ and $\mathbf{y}$, our objective is to infer the value and uncertainty of $f^* = f(\mathbf{x}^*)$ where $\mathbf{x}^* \in \mathbb{R}^D$ is called the test data. The joint Gaussian distribution can be expressed by,
\begin{equation}\label{equ:priori-state}
    \left[\begin{array}{ c }
        \mathbf{y}\\
        f^{*}
        \end{array}\right] \sim \mathcal{N}\left(\mathbf{0}_{N+1} ,\left[\begin{array}{ c c }
            \mathbf{K_\SSmy} (\mathbf{X} ,\mathbf{X} )+\sigma_{0}^{2}\mathbf{I}_{N} & \mathbf{K_\SSmy} (\mathbf{X} ,\mathbf{x}^{*} )\\
            \mathbf{K_\SSmy} (\mathbf{x}^{*} ,\mathbf{X} ) & \mathbf{K_\SSmy} (\mathbf{x}^{*} ,\mathbf{x}^{*} )
    \end{array}\right]\right).
\end{equation}
Here, the covariance matrix of the training data denoted by $\mathbf{K_\SSmy}(\mathbf{X}, \mathbf{X})\in \mathbb{R}^{N\times N}$ is defined such that $\mathbf{K_\SSmy}(\mathbf{X}, \mathbf{X})_{i, j} = k\left(\mathbf{x}_i ,\mathbf{x}_j;\boldsymbol{\theta}\right)$, where the subscript ``$\SSmy$'' represents the covariance between states. $\mathbf{K_\SSmy}(\mathbf{X} , \mathbf{x}^{*} ) \in \mathbb{R}^{N}$ represents the cross-covariance matrix between the training data $\mathbf{X}$ and the test data $\mathbf{x}^{*}$ with $\mathbf{K_\SSmy}(\mathbf{X} , \mathbf{x}^{*} )_{i} = k\left(\mathbf{x}_i ,x^*;\boldsymbol{\theta}\right)$. $\mathbf{K_\SSmy}(\mathbf{x}^{*}, \mathbf{X}) = \mathbf{K_\SSmy}(\mathbf{X} , \mathbf{x}^{*} )^T$ and $\mathbf{K_\SSmy}(\mathbf{x}^*, \mathbf{x}^*) = k\left(\mathbf{x}^*, \mathbf{x}^*;\boldsymbol{\theta}\right)$.

The posterior distribution, which represents the conditional distribution of $f^{*}$ given $\mathbf{X}$, $\mathbf{y}$, $\mathbf{x}^{*}$, and $\boldsymbol{\theta}$, can be expressed by
\begin{equation*}
    f^{*} \big| \mathbf{X} ,\mathbf{y} ,\mathbf{x}^{*} ,\boldsymbol{\theta} \sim \mathcal{N}\left(\overline{f^{*}} ,\Var(f^{*})\right),
\end{equation*}
where the mean and variance of the posterior distribution are calculated by
\begin{equation}\label{equ:pred-state}
    \begin{cases}
        \overline{f^{*}} =\mathbb{E} [f^{*} | \mathbf{X}, \mathbf{y}, \mathbf{x}^{*} ,\boldsymbol{\theta} ]=\mathbf{K_\SSmy} (\mathbf{x}^{*} ,\mathbf{X} )\left[\mathbf{K_\SSmy} (\mathbf{X} ,\mathbf{X} )+\sigma _{0}^{2}\mathbf{I}_N\right]^{-1}\mathbf{y} ,\\
        \Var(f^{*})=\mathbf{K_\SSmy} (\mathbf{x}^{*}, \mathbf{x}^{*} )-\mathbf{K_\SSmy} (\mathbf{x}^{*} ,\mathbf{X} )\left[\mathbf{K_\SSmy} (\mathbf{X} ,\mathbf{X} )+\sigma _{0}^{2}\mathbf{I}_N\right]^{-1}\mathbf{K_\SSmy} (\mathbf{X}, \mathbf{x}^{*} ).
    \end{cases}
\end{equation}

\subsubsection{Inferring the partial derivatives of state variables}
Since differentiation is a linear operator, the derivative of a Gaussian process remains a Gaussian process \cite{rasmussen2006gaussian}. Beyond inferring the mean and variance of the states data $f^* = f(\mathbf{x}^*)$, GPR enables the derivation of the partial derivatives, such as the first-order partial derivative with respect to the $j$-th component $\mathbf{x}_j^*$, denoted by $(\partial f^{*})_{j} =\frac{\partial f\left(\mathbf{x}^{*}\right)}{\partial \mathbf{x}_{j}^{*}}$. Similar to the formulas in \eqref{equ:priori-state}, a joint Gaussian distribution encompassing the derivatives and states data can be expressed by,
\begin{equation*}
    \left[\begin{array}{ c }
        \mathbf{y}\\
        (\partial f^{*})_{j}
        \end{array}\right] \sim \mathcal{N}\left(\mathbf{0} ,\left[\begin{array}{ c c }
            \mathbf{K_\SSmy}(\mathbf{X} ,\mathbf{X} )+\sigma _{0}^{2} \mathbf{I}_N & \mathbf{K_\SD}(\mathbf{X} ,\mathbf{x}^{*} )\\
            \mathbf{K_\DS}(\mathbf{x}^{*} ,\mathbf{X} ) & \mathbf{K_\DD}(\mathbf{x}^{*}, \mathbf{x}^{*} )
    \end{array}\right]\right),
\end{equation*}
where the subscript ``S'' and ``D'' represent ``State'' and ``Derivative'', respectively. $\mathbf{K_\SD}(\mathbf{X} , \mathbf{x}^{*} ) \in \mathbb{R}^{N}$ and $\mathbf{K_\DS}(\mathbf{x}^{*}, \mathbf{X}) = \mathbf{K_\SD}(\mathbf{X} , \mathbf{x}^{*} )^T$ represent two covariance matrices between the training data $\mathbf{X}$ and the test data $\mathbf{x}^{*}$, where $\mathbf{K_\SD}(\mathbf{x} , \mathbf{x}^{*} )_{i, 1} = \frac{\partial k\left(\mathbf{x}, \mathbf{x}';\boldsymbol{\theta}\right)}{\partial \mathbf{x}_j'}\big|_{\mathbf{x}=\mathbf{x}_{i} ,\mathbf{x}'=\mathbf{x}^{*}}$. Similarly, $\mathbf{K_\DD}(\mathbf{x}^{*} , \mathbf{x}^{*} ) = \frac{\partial^2 k\left(\mathbf{x}, \mathbf{x}';\boldsymbol{\theta}\right)}{\partial \mathbf{x}_j\partial \mathbf{x}'_j}\big|_{\mathbf{x}=\mathbf{x}^*, \mathbf{x}'=\mathbf{x}^{*}}$.

The conditional distribution of $(\partial f^{*})_{j}$ given $\mathbf{X} ,\mathbf{y} ,\mathbf{x}^{*}, \boldsymbol{\theta}$, can be given by,
\begin{equation*}
    (\partial f^{*})_{j} \big| \mathbf{X} ,\mathbf{y} ,\mathbf{x}^{*}, \boldsymbol{\theta} \sim \mathcal{N}\left(\overline{(\partial f^{*})}_{j} ,\Var((\partial f^{*})_{j})\right),
\end{equation*}
where the mean and variance of the posterior distribution are calculated by
\begin{equation}\label{equ:pred-der}
    \begin{cases}
        \overline{(\partial f^{*})}_{j} = \mathbb{E} [(\partial f^{*})_{j} | \mathbf{X}, \mathbf{y}, \mathbf{x}^{*}, \boldsymbol{\theta}]=\mathbf{K_\DS} (\mathbf{x}^{*} ,\mathbf{X} )\left[\mathbf{K_\SSmy} (\mathbf{X} ,\mathbf{X} )+\sigma _{0}^{2}\mathbf{I}_N\right]^{-1}\mathbf{y} ,\\
        \Var((\partial f^{*})_{j}) = \mathbf{K_\DD} (\mathbf{x}^{*}, \mathbf{x}^{*} )-\mathbf{K_\DS} (\mathbf{x}^{*} ,\mathbf{X} )\left[\mathbf{K_\SSmy} (\mathbf{X} ,\mathbf{X} )+\sigma _{0}^{2}\mathbf{I}_N\right]^{-1}\mathbf{K_\SD} (\mathbf{X}, \mathbf{x}^{*} ).
    \end{cases}
\end{equation}

The inference of the partial derivatives is utilized for the terms such as $u_t$ in ODEs, as well as $u_t, u_{x}, u_{y}, u_{xx}, u_{yy}$ in PDEs. Here, we focus solely on demonstrating how to infer the first-order partial derivatives. Further details regarding the derivation of the first-order and higher-order partial derivatives of the SE kernel are given in \ref{sec:app-der-SE}.

\subsubsection{Training the hyperparameters in GP}
The marginal likelihood at $\mathbf{X}$ can be described by
\[
    p(\mathbf{y} |\mathbf{X}) =\int p(\mathbf{y} |\mathbf{f} ,\mathbf{X}) p(\mathbf{f} |\mathbf{X}) \mathrm{d}\mathbf{f},
\]
where the priori $p(\mathbf{f} |\mathbf{X})$ is a Gaussian distribution, \ie, $\mathbf{f} |\mathbf{X} \sim \mathcal{N}(\mathbf{0} ,\mathbf{K_{\SSmy }} (\mathbf{X} ,\mathbf{X} ))$ \cite{rasmussen2006gaussian}. Meanwhile, due to the existence of noise in GP and $p(\mathbf{y} |\mathbf{f} ,\mathbf{X}) =p(\mathbf{y} |\mathbf{f})$, we have $\mathbf{y} |\mathbf{f} ,\mathbf{X} \sim \mathcal{N}\left(\mathbf{f} ,\sigma _{0}^{2}\mathbf{I}_{N}\right)$. Then, we obtain $\mathbf{y} |\mathbf{X} \sim \mathcal{N}\left(\mathbf{0} ,\mathbf{K_{\SSmy }} (\mathbf{X} ,\mathbf{X} )+\sigma _{0}^{2}\mathbf{I}_{N}\right)$. The value of hyperparameters $\boldsymbol{\theta}$ in the kernel of GP, as well as the noise variance $\sigma_{0}^{2}$ are determined by maximizing the likelihood function of $\mathbf{y} |\mathbf{X}$, which is equivalent to minimizing the negative log marginal likelihood given by,
\begin{equation}\label{equ:nlml}
    \mathcal{L}_{\text{GP}}\left(\boldsymbol{\theta} ,\sigma _{0}^{2} ;\boldsymbol{y} ,\mathbf{X}\right) =\frac{1}{2}\boldsymbol{y}^{\mathrm{T}}\left(\mathbf{K_\SSmy} (\mathbf{X} ,\mathbf{X} )+\sigma _{0}^{2}\mathbf{I}_{N}\right)^{-1}\boldsymbol{y} +\frac{1}{2}\log\left| \mathbf{K_\SSmy} (\mathbf{X} ,\mathbf{X} )+\sigma _{0}^{2}\mathbf{I}_{N}\right| +\frac{N}{2}\log (2\pi).
\end{equation}

\subsection{GP-SINDy for single-fidelity data}
In this part, w propose the GP-SINDy algorithm, which utilizes $\mathbf{X}$ and the noisy observed data and $\mathbf{U}$ to compute the sparse solution $\mathbf{C}$ for the WLS problem given by \eqref{equ:SINDy-opt-2}. GP-SINDy involves four steps,
\vspace{-.1cm}
\begin{enumerate}
    \item Constructing the GP surrogate model based on observed data $\mathbf{X}, \mathbf{U}$;
    \item Inferring the states variable, and their partial derivatives;
    \item Assembling the derivatives matrix $\dot{\mathbf{U}}_{i}$ and function library matrix $\boldsymbol{\Phi}$;
    \item Resolving the WLS problem \eqref{equ:SINDy-opt-2}.
\end{enumerate}
\vspace{-.1cm}

Since the dimension of the state variable $\mathbf{u}$ is $d$ and the single output GP is employed, it becomes imperative to compute the sparse coefficients for each individual dimension of $\mathbf{u}$. Now we need to construct $d$ independent GPR models in order to extract information for each dimension, including the states and the associated derivatives. The $i$-th GPR model is constructed based on training data $\{\mathbf{X}, \mathbf{U}_i\}$ and the kernel function $k\left(\mathbf{x}, \mathbf{x}^{\prime } ;\theta_i\right)$, where the optimal hyperparameters $\theta_i$ and noise variance $\sigma _{i}^{2}$ are determined by minimizing equation \eqref{equ:nlml}.

Then, we perform the inference at $N'$ test inputs (denoted by $\mathbf{X}_p \in\mathbb{R}^{D\times N'}$) using equation \eqref{equ:pred-state} and \eqref{equ:pred-der}. This inference involves three parts: the states, the posterior mean and variance of partial derivatives w.r.t. time denoted by $\overline{\mathbf{U}}_i$, $\overline{\dot{\mathbf{U}}}_i$, and $\Var(\dot{\mathbf{U}}_i)$, respectively. Here, $\Var(\dot{\mathbf{U}_i})$ is a matrix with size $N'\times N'$. For PDEs, we also need to calculate the partial derivatives w.r.t. the spatial variables denoted by $\overline{\mathbf{V}}_i$, such as first-order and second-order partial derivatives matrices $\overline{\mathbf{V}}_i$ $\text{(order=1)}$ and $\overline{\mathbf{V}}_i$ $\text{(order=2)}$. Typically, to ensure accuracy, the size of the prediction is set to be greater than the size of the training data, \ie, $N' \geq N$.

\begin{remark}
    Standard GPR has cubic computational complexity when learning the hyperparameters and predicting target values. To alleviate this computational burden, various algorithms have been developed \cite{snelson2005sparse, liu2020when}. Nonetheless, to keep a clear structure of this paper, we utilize conventional GPR with a constrained dataset size to maintain the computational tractability. Furthermore, at the stage of hyperparameters optimization, subsampling of the training data is an optional technique to considerably diminish the runtime.
\end{remark}

Let $\overline{\mathbf{U}} = [\overline{\mathbf{U}}_1, \overline{\mathbf{U}}_2, \cdots, \overline{\mathbf{U}}_d]$, $\overline{\mathbf{V}} = [\overline{\mathbf{V}}_1, \overline{\mathbf{V}}_2, \cdots, \overline{\mathbf{V}}_d]$. In the function library of PDEs, the term $\mathbf{u}$, $\mathbf{u}_{x}$, and $\mathbf{u}_{xx}$ are approximated by the state matrix $\overline{\mathbf{U}}$, the first-order and second-order partial derivatives matrices $\overline{\mathbf{V}}$ $\text{(order=1)}$ and $\overline{\mathbf{V}}$ $\text{(order=2)}$, respectively. The nonlinear term $\mathbf{u}\mathbf{u}_{x}$ is approximated by element-wise product of two matrices $\overline{\mathbf{U}}$ and $\overline{\mathbf{V}}$ $\text{(order=1)}$. Hence, in equation \eqref{equ:SINDy-opt-2}, the derivatives matrix $\dot{\mathbf{U}}_{i}$ is computed by the posterior mean $\overline{\dot{\mathbf{U}}}_i$, and the function library matrix $\boldsymbol{\Phi}$ is constructed using $\overline{\mathbf{U}}$ and $\overline{\mathbf{V}}$, \ie, $\boldsymbol{\Phi}  = \boldsymbol{\Phi} (\overline{\mathbf{U}}, \overline{\mathbf{V}})$. As a result, we obtain the following linear equation to be solved:
\begin{equation}\label{equ:wls-noise-pre}
    \overline{\dot{\mathbf{U}}}_i = \boldsymbol{\Phi }\mathbf{C}_{i} + \boldsymbol{\varepsilon}_i,
\end{equation}
where $\boldsymbol{\varepsilon}_i \sim \mathcal{N}(\mathbf{0}, \boldsymbol{\Sigma}_i)$ accounts for the noise and embodies the uncertainty associated with the derivative data $\overline{\dot{\mathbf{U}}}_i$. We can rewrite this equation as a WLS problem given by,
\begin{equation}\label{equ:wls-noise}
    \min_{\mathbf{C}_i} \left(\overline{\dot{\mathbf{U}}}_i -\boldsymbol{\Phi }\mathbf{C}_{i}\right)^{T}\mathbf{W}_{i}\left(\overline{\dot{\mathbf{U}}}_i -\boldsymbol{\Phi }\mathbf{C}_{i}\right).
\end{equation}
Before computing the sparse solution of $\mathbf{C}_{i}$, the following lemma illustrates the WLS estimator of equation \eqref{equ:wls-noise}.

\begin{lemma}[Gauss-Markov Theorem~\cite{hansen2022modern}]\label{lemma:gm}
    For the least squares problem 
    \[
    \min_{\mathbf{C}_i} \left(\overline{\dot{\mathbf{U}}}_i -\boldsymbol{\Phi }\mathbf{C}_{i}\right)^{T}\left(\overline{\dot{\mathbf{U}}}_i -\boldsymbol{\Phi }\mathbf{C}_{i}\right),
    \]
arising from Equation~\eqref{equ:wls-noise-pre} with $\boldsymbol{\Sigma}_i = \sigma^2\mathbf{I}$, the best linear unbiased estimator (BLUE) is given by $\hat{\mathbf{C}_{i}} = (\boldsymbol{\Phi}^T\boldsymbol{\Phi})^{-1}\boldsymbol{\Phi}^T\overline{\dot{\mathbf{U}}}_i$.
\end{lemma}

\begin{theorem}\label{thm:WLS}
    For the weighted least squares problem~\eqref{equ:wls-noise} associated with Equation~\eqref{equ:wls-noise-pre}, the optimal solution is given by
    \[
        \hat{\mathbf{C}_{i}} = (\boldsymbol{\Phi}^T\mathbf{W}_i\boldsymbol{\Phi})^{-1}\boldsymbol{\Phi}^T\mathbf{W}_i\overline{\dot{\mathbf{U}}}_i.
    \]
For the special case where $\mathbf{W}_i=\boldsymbol{\Sigma}_i^{-1}$, $\hat{\mathbf{C}_{i}}$ becomes the best linear unbiased estimator (BLUE).
\end{theorem}
\begin{proof}
    The optimality condition of~\eqref{equ:wls-noise} is given by
    \[
        \frac{\partial }{\partial \mathbf{C}_{i}} \left(\overline{\dot{\mathbf{U}}}_i -\boldsymbol{\Phi }\mathbf{C}_{i}\right)^{T}\mathbf{W}_{i}\left(\overline{\dot{\mathbf{U}}}_i -\boldsymbol{\Phi }\mathbf{C}_{i}\right) = 2\boldsymbol{\Phi }^T \mathbf{W}_{i} \left(\overline{\dot{\mathbf{U}}}_i -\boldsymbol{\Phi }\mathbf{C}_{i}\right) =0,
    \]
    and we obtain the optimal solution 
    \[
        \hat{\mathbf{C}_{i}} = (\boldsymbol{\Phi}^T\mathbf{W}_i\boldsymbol{\Phi})^{-1}\boldsymbol{\Phi}^T\mathbf{W}_i\overline{\dot{\mathbf{U}}}_i.
    \]
Since $\mathbb{E}\left(\overline{\dot{\mathbf{U}}}_i\right) = \mathbb{E}\left(\boldsymbol{\Phi }\mathbf{C}_{i} + \boldsymbol{\varepsilon}_i\right) = \boldsymbol{\Phi }\mathbf{C}_{i}$, we have
    \[
        \begin{aligned}
            \mathbb{E}\left(\hat{\mathbf{C}_{i}}\right) & = (\boldsymbol{\Phi}^T\mathbf{W}_i\boldsymbol{\Phi})^{-1}\boldsymbol{\Phi}^T\mathbf{W}_i\mathbb{E}\left(\overline{\dot{\mathbf{U}}}_i\right)\\
             & = (\boldsymbol{\Phi}^T\mathbf{W}_i\boldsymbol{\Phi})^{-1}\boldsymbol{\Phi}^T\mathbf{W}_i \boldsymbol{\Phi }\mathbf{C}_{i}\\
             & = \mathbf{C}_{i},
        \end{aligned}
    \]
    which indicates that $\hat{\mathbf{C}_{i}}$ is unbiased.

    For the positive semi-definite matrix $\boldsymbol{\Sigma}_i$, denote its Cholesky decomposition by $\boldsymbol{\Sigma}_i = \mathbf{L}\mathbf{L}^T$. Then, Equation~\eqref{equ:wls-noise-pre} can be rewritten by,
    \[
        \mathbf{L}^{-1}\overline{\dot{\mathbf{U}}}_i = \mathbf{L}^{-1}\boldsymbol{\Phi }\mathbf{C}_{i} + \mathbf{L}^{-1}\boldsymbol{\varepsilon}_i,
    \]
    where $\mathbf{L}^{-1}\boldsymbol{\varepsilon}_i \sim \mathcal{N}(\mathbf{0}, \mathbf{I})$. The least square problem 
    \[
    \min_{\mathbf{C}_i} \left(\mathbf{L}^{-1}\overline{\dot{\mathbf{U}}}_i -\mathbf{L}^{-1}\boldsymbol{\Phi }\mathbf{C}_{i}\right)^{T}\left(\mathbf{L}^{-1}\overline{\dot{\mathbf{U}}}_i -\mathbf{L}^{-1}\boldsymbol{\Phi }\mathbf{C}_{i}\right),
    \]
    which is identical to Equation~\eqref{equ:wls-noise} when $\mathbf{W}_i=\boldsymbol{\Sigma}_i^{-1}$. According to Lemma~\ref{lemma:gm}, the BLUE is given by
    \[
        \begin{aligned}
            \hat{\mathbf{C}_{i}} & = \left(\left(\mathbf{L}^{-1}\boldsymbol{\Phi}\right)^T\mathbf{L}^{-1}\boldsymbol{\Phi}\right)^{-1}\left(\mathbf{L}^{-1}\boldsymbol{\Phi}\right)^T\mathbf{L}^{-1}\overline{\dot{\mathbf{U}}}_i\\
             & = (\boldsymbol{\Phi}^T(\mathbf{L}\mathbf{L}^T)^{-1}\boldsymbol{\Phi})^{-1}\boldsymbol{\Phi}^T(\mathbf{L}\mathbf{L}^T)^{-1}\overline{\dot{\mathbf{U}}}_i\\
             & = (\boldsymbol{\Phi}^T\boldsymbol{\Sigma}_i^{-1}\boldsymbol{\Phi})^{-1}\boldsymbol{\Phi}^T\boldsymbol{\Sigma}_i^{-1}\overline{\dot{\mathbf{U}}}_i.
        \end{aligned}
    \]
\end{proof}

The Sequential Threshold Ridge regression (STRidge) algorithm introduced by \citet{brunton2016discovering, rudy2017datadriven} offers a powerful approach to seek the sparse solution of equation \eqref{equ:wls-noise}. Unlike other methods that relax the $\ell_0$ regularization to $\ell_1$ regularization, such as the least absolute shrinkage and selection operator (LASSO), STRidge enforces sparsity by applying thresholds iteratively. Based on STRidge, we propose the Sequential Threshold for Weighted Least Squares (STWLS) described by Algorithm \ref{alg:STWLS}, which incorporates the WLS estimator into the sparsity-promoting process. In order to keep the clarity of Algorithm \ref{alg:STWLS}, we replace $\overline{\dot{\mathbf{U}}}_i$ and $\mathbf{W}_i$ by the notations $\mathbf{Z}$ and $\mathbf{W}$, respectively.

\begin{algorithm}[ht] 
	\caption{Sequential Threshold for Weighted Least Squares (STWLS)}
	\label{alg:STWLS}
    {\bf Input:} $\boldsymbol{\Phi} \in\mathbb{R}^{N'\times N_f}$, $\mathbf{Z} \in\mathbb{R}^{N'}$, $\mathbf{W}\in\mathbb{R}^{N_f\times N_f}$, $\Lambda (\lambda \geq 0, \forall \lambda \in \Lambda)$, $\eta > 0$, $K$, $J$.
	\begin{algorithmic}[1]
        \State {\bf Initialization:} $\tau = 0$, $\mathcal{I}_0 = \{1, \cdots, N_f\}$, $\mathbf{c}^* = \boldsymbol{\Phi}^\dagger \mathbf{Z}$, $\mathcal{L}^* = (\boldsymbol{\Phi}\mathbf{c}^*-\mathbf{Z})^T\mathbf{W}(\boldsymbol{\Phi}\mathbf{c}^*-\mathbf{Z}) +\eta \Vert \mathbf{c}^* \Vert_0$
        \For{$\lambda \in \Lambda$}
            \For{$k = 1$ to $K$}
                \State $\mathcal{I} = \mathcal{I}_0$, $\widetilde{\boldsymbol{\Phi}} = \boldsymbol{\Phi}$, $\widetilde{\mathbf{W}} = \mathbf{W}$
                \For{$j = 1$ to $J$}
                    \State $\mathbf{c}[\mathcal{I}]= (\widetilde{\boldsymbol{\Phi}}^T \widetilde{\mathbf{W}} \widetilde{\boldsymbol{\Phi}}+\lambda \mathbf{I})^{-1}(\widetilde{\boldsymbol{\Phi}}^T \widetilde{\mathbf{W}}\mathbf{Z})$
                    \State $\mathcal{I} = \{j:| \mathbf{c} [j]| \geq \tau \}$, $\mathbf{c}[\mathcal{I}_0 \setminus \mathcal{I}] = 0$
                    \State $\widetilde{\boldsymbol{\Phi}} = \widetilde{\boldsymbol{\Phi}}[:, \mathcal{I}]$, $\widetilde{\mathbf{W}} = \widetilde{\mathbf{W}}[\mathcal{I}, \mathcal{I}]$
                \EndFor
                \State $\mathbf{c}[\mathcal{I}]= (\widetilde{\boldsymbol{\Phi}}^T \widetilde{\mathbf{W}} \widetilde{\boldsymbol{\Phi}})^{-1}(\widetilde{\boldsymbol{\Phi}}^T \widetilde{\mathbf{W}}\mathbf{Z})$, $\mathbf{c}[\mathcal{I}_0 \setminus \mathcal{I}] = 0$
                \State $\mathcal{L} = (\boldsymbol{\Phi}\mathbf{c}-\mathbf{Z})^T\mathbf{W}(\boldsymbol{\Phi}\mathbf{c}-\mathbf{Z}) +\eta \Vert \mathbf{c} \Vert_0$
                \If{$\mathcal{L} \leq \mathcal{L}^*$}
                    \State $\mathcal{L}^* = \mathcal{L}$, $\mathbf{c}^*=\mathbf{c}$
                \EndIf
                \State $\tau = 1.05\min_{i}\{|c^*_i|: c^*_i\neq 0\}$
            \EndFor
        \EndFor
    \end{algorithmic}
    {\bf Output:} Sparse solution $\mathbf{c}^*$.
\end{algorithm}

In STWLS, $\Lambda$ is the set that contains all candidate values of $\lambda$. For a fixed $\lambda \in \Lambda$, there are two loops, the outer loop is to determine the suitable threshold $\tau$, while the inner loop seeks the sparse solution under the specified threshold iteratively. It is noteworthy that in STWLS, $\lambda$ and $\eta$ are two key parameters for sparse solutions. A series of ridge regression techniques are employed to identify the support function terms in the inner loop, wherein $\lambda$ serves as the regularization parameter. Non-zero $\lambda$ is employed to highlight the correct feature terms. Different from the standard ridge regression, $\lambda$ in STWLS merely has impact on the support function terms (line 6 in Algorithm \ref{alg:STWLS}), which makes an indirect contribution to the final coefficient matrix. Another key parameter $\eta$ is utilized to decide whether to accept the sparse solution from the inner loop or not, which makes the trade-off between the residual of linear equations and the sparsity of the solution. When $\eta$ is smaller, the sparsity loss term is less penalized, which leads to the defectively sparse solution. Conversely, if it is too large, the solution exhibits an exceptionally high degree of sparsity since the sparsity loss term becomes dominant.

\begin{remark}
    In the numerical experiments of Section \ref{sec:Numerical}, the final outcomes are little sensitive to $\eta$. In practice, $\lambda$ is typically difficult to choose, especially when the training data is limited, small $\lambda$ would result in an excess of the function terms while larger values would yield much fewer terms. Analogous phenomenon also present in \cite{cortiella2021sparse}, where corner point criterion in Pareto curve is exploited to balance the fitting and sparsity of the solution. Here, we choose $\lambda$ in a naive but effective manner by setting a candidate set $\Lambda$. Numerical results in Section \ref{sec:Numerical} demonstrate its effectiveness.
\end{remark}

\begin{remark}
    In Algorithm \ref{alg:STWLS}, we address the WLS problem in the context of sparse identification of nonlinear systems, where the variance of noise is approximated by the posterior variance of temporal partial derivatives in GPR. Similar techniques can be found in related literatures, for example, the generalized least squares problem \cite{messenger2021weak, chen2023datadriven}, wherein the covariance matrix (corresponding to the inverse of the aforementioned weight matrix) is approximated through the residual analysis of methods in \cite{messenger2021weak, chen2023datadriven}.
\end{remark}

Finally, for the WLS problem \eqref{equ:wls-noise}, STWLS algorithm (Algorithm \ref{alg:STWLS}) is employed to derive the sparse solution $\mathbf{C}_i$. We leverage the posterior variance of the partial derivative w.r.t. time to approximate the covariance matrix of $\boldsymbol{\varepsilon}_i$, \ie, the weight matrix $\mathbf{W}_i = \left(\diag\left(\Var(\dot{\mathbf{U}_i})\right)\right)^{-1}$. And our Gaussian Process based SINDy (GP-SINDy) algorithm is summarized in Algorithm \ref{alg:GPSINDy}.

\begin{algorithm}[H] 
	\caption{Gaussian Process based SINDy (GP-SINDy)}
	\label{alg:GPSINDy}
	{\bf Input:} Data $\mathcal{D} =\left\{\mathbf{X} \in \mathbb{R}^{D\times N}, \mathbf{U} \in \mathbb{R}^{N\times d}\right\}$, test input $\mathbf{X}_p \in\mathbb{R}^{D\times N'}$, parameters of STWLS $\Lambda$, $\eta$, $K$, $J$.
	\begin{algorithmic}[1]
        \For{$i \in \{ 1, \cdots, d\}$}
            \State Construct the GPR model with $\{\mathbf{X}, \mathbf{U}_i\}$, and learn the optimal hyperparameters, $\theta_i$ and $\sigma^{2}_i$.
            \State Infer the posterior $\overline{\mathbf{U}}_i$, $\overline{\dot{\mathbf{U}}}_i$, $\Var\left(\dot{\mathbf{U}_i}\right)$, and $\overline{\mathbf{V}}_i$ at test input $\mathbf{X}_p$.
            \State Compute the weight matrix $\mathbf{W}_i = \left(\diag\left(\Var(\dot{\mathbf{U}_i})\right)\right)^{-1}$.
        \EndFor
        \State Assemble the data matrices, $\overline{\mathbf{U}} = [\overline{\mathbf{U}_1}, \overline{\mathbf{U}_2}, \cdots, \overline{\mathbf{U}_d}]$, $\overline{\mathbf{V}} = [\overline{\mathbf{V}_1}, \overline{\mathbf{V}_2}, \cdots, \overline{\mathbf{V}_d}]$.
        \State Compute the function library matrix $\boldsymbol{\Phi}  = \boldsymbol{\Phi} (\overline{\mathbf{U}}, \overline{\mathbf{V}})$.
        \For{$i \in \{ 1, \cdots, d\}$}
            \State $\mathbf{C}[:, i]$ = STWLS($\boldsymbol{\Phi}$, $\overline{\dot{\mathbf{U}}}_i$, $\mathbf{W}_{i}$, $\Lambda$, $\eta$, $K$, $J$).
        \EndFor
	\end{algorithmic}
	{\bf Output:} Sparse coefficient matrix $\mathbf{C}$.
\end{algorithm}

\subsection{MFGP-SINDy for multi-fidelity data}
In this part, we will briefly review the nonlinear autoregressive multi-fidelity GP regression (MFGP) \cite{perdikaris2017nonlinear} to assimilate the nonlinear information from the MF data into the inference of the HF model. Let $\mathcal{D}^l = \curly{(\mathbf{x}^l_{i} ,y^l_{i})\Big | ,i=1,\cdots ,N^l}$ with $l=1, \cdots, L$ denote the observed MF data, which satisfies,
\begin{equation*}
    y^l_{i} = f^l(\mathbf{x}^l_{i}) + \varepsilon^l_i,
\end{equation*}
where $\mathbf{x}^l_{i}\in \Omega^l \subset \mathbb{R}^D$ represents the input, $y^l_{i}\in\mathbb{R}$ is the noisy output of level $l$, and the noise term $\varepsilon^l_i$ is i.i.d. at each level with Gaussian distribution $\mathcal{N}\left(0, \left(\sigma_\MF^l\right)^2\right)$. We rewrite the training data as $\mathcal{D}^l = \curly{(\mathbf{X}^l, \mathbf{y}^l)}$ with $l=1, \cdots, L$ where $\mathbf{X}^l =[\mathbf{x}_{1}^l, \mathbf{x}_{2}^l, \cdots, \mathbf{x}_{N^l}^l] \in \mathbb{R}^{D\times {N^l}}$ and $\mathbf{y}^l = [y_{1}^l, y_{2}^l, \cdots, y_{N^l}^l]^T\in \mathbb{R}^{N^l}$.
Here, $\mathcal{D}^L$ corresponds to the HF data, and $\mathcal{D}^l$ for $(l=1, \cdots, L-1)$ are the LF data sorted by increasing the level of accuracy. For the function $f^1$, it can be described by a Gaussian process, \ie,
\[
    f^1(\mathbf{x}) \sim \mathcal{GP}\left(\overline{f^1}(\mathbf{x}) ,k^1\left( \mathbf{x},\mathbf{x}^{\prime } ;\theta^1\right)\right),
\]
where $k^1\left( \mathbf{x},\mathbf{x}^{\prime } ;\theta^1\right)$ is an SE kernel. We use the notation $\Cov(\cdot, \cdot)$ to denote the covariance between two random variables. Assume that the $f^1, f^2, \cdots, f^L$ follow the Markov property,
\[
    \Cov\left(f^l(\mathbf{x}), f^{l-1}(\mathbf{x}') \big | f^{l-1}(\mathbf{x})\right) = 0, \forall \mathbf{x} \neq \mathbf{x}', l=2,\cdots, L,
\]
which means that given the value of $f^{l-1}(\mathbf{x})$, we can learn nothing more about $f^l(\mathbf{x})$ from any other information $f^{l-1}(\mathbf{x}')$, for $\mathbf{x} \neq \mathbf{x}'$ \cite{kennedy2000predicting, perdikaris2017nonlinear}. Hence, the correlation between two levels of models, $f^{l-1}$ and $f^l$, can be described by,
\begin{equation}\label{equ:mfgp-1}
    f^{l}(\mathbf{x}) = z^{l}(f^{l-1}(\mathbf{x})) + \delta^l(\mathbf{x}),
\end{equation}
where $z^{l}:\mathbb{R}\rightarrow \mathbb{R}$ is an unknown nonlinear function and $\delta^l(\mathbf{x})$ is a Gaussian process independent of $z^{l}$. Specifically, when $z^l$ is linear, equation \eqref{equ:mfgp-1} degenerates to a linear autoregressive structure described in previous works \cite{kennedy2000predicting, legratiet2014recursive}.

However, equation \eqref{equ:mfgp-1} involves a complicated problem that results in huge computational complexity. To address this issue, we can approximate $z^{l}$ with a Gaussian process, which embeds the LF model $f^{l-1}(\mathbf{x})$ into the higher fidelity model $f^{l}(x)$. This is achieved by substituting $f^{l-1}(\mathbf{x})$ with its posterior estimation $\overline{{f}^{l-1}}(\mathbf{x})$. As a result, the correlation \eqref{equ:mfgp-1} can be transformed as:
\begin{equation*}\label{equ:mfgp-2}
    f^{l}(\mathbf{x}) = g^{l}(\mathbf{x}, \overline{{f}^{l-1}}(\mathbf{x})),
\end{equation*}
where $g^{l}$ is a nonlinear function, which maps from a $(D+1)$-dimensional subspace to a scalar in $\mathbb{R}$. It can be represented by a Gaussian process, and prior distribution of the $l$-th level function $f^l(\mathbf{x})$ is
\[
    f^l(\mathbf{x}) \sim \mathcal{GP}\left(\overline{f^l}(\mathbf{x}), k^l\left(\left(\mathbf{x}, \overline{{f}^{l-1}}(\mathbf{x})\right), \left(\mathbf{x}', \overline{{f}^{l-1}}(\mathbf{x}')\right);\theta^{l}\right)\right),
\]
where $k^l\left(\left(\mathbf{x}, \overline{{f}^{l-1}}(\mathbf{x})\right), \left(\mathbf{x}', \overline{{f}^{l-1}}(\mathbf{x}')\right);\theta^{l}\right)$ is the corresponding kernel function and $\theta^{l}$ is its hyperparameter. Here, $\delta^l(\mathbf{x})$ can be considered implicitly due to the first input argument $\mathbf{x}$.

However, $\mathbf{x}$ and $\overline{{f}^{l-1}}(\mathbf{x})$ belong to different spaces. For the sake of derivation of partial derivative of the kernel, we apply the separation of variables to this kernel and define its decomposition structure by,
\begin{equation}\label{equ:mfgp-kernel}
    k^{l}\left(\left(\mathbf{x} ,\overline{f^{l-1}} (\mathbf{x} )\right) ,\left(\mathbf{x} ',\overline{f^{l-1}} (\mathbf{x} ')\right) ;\theta ^{l}\right) =k_{\rho }(\mathbf{x} ,\mathbf{x} ';\theta _{\rho }) k_{f}\left(\overline{f^{l-1}} (\mathbf{x} ),\overline{f^{l-1}} (\mathbf{x} ');\theta _{f}\right) +k_{\delta }(\mathbf{x} ,\mathbf{x} ';\theta _{\delta }),
\end{equation}
where $k_\rho$, $k_f$, and $k_\delta$ are three SE kernels. In this paper, this kernel is referred to as the MFGP kernel. It bridges the gap between the input $\mathbf{x}$ and the estimated LF function values $\overline{{f}^{l-1}}(\mathbf{x})$ in a flexible form.

\subsubsection{MFGP construction}
In this paper, we focus on coping with the bi-fidelity data, \ie, $L=2$, wherein $f^1(\mathbf{x}) \sim \mathcal{GP}\left(\overline{f^1}(\mathbf{x}) ,k^1\left( \mathbf{x},\mathbf{x}^{\prime } ;\theta^1\right)\right)$ and $f^2(\mathbf{x}) \sim \mathcal{GP}\left(\overline{f^2}(\mathbf{x}), k^2\left(\left(\mathbf{x}, \overline{{f}^{1}}(\mathbf{x})\right), \left(\mathbf{x}', \overline{{f}^{1}}(\mathbf{x}')\right);\theta^{2}\right)\right)$ represent the LF and HF GP models. We use the abbreviated notations $\mathcal{GP}^{1}$ and $\mathcal{GP}^{2}$ to represent them, respectively. Analogous to GPR, MFGP consists of two stages: training and inferring, and this bi-fidelity GP scheme can be readily extended to deal with even higher fidelity levels data.

In the training stage, we firstly construct the $\mathcal{GP}^{1}$ model using the LF data $\mathcal{D}^1 = \curly{(\mathbf{x}^1, \mathbf{y}^1)}$ and a SE kernel $k^1(\mathbf{x}, \mathbf{x}';\theta^1)$. The hyperparameter $\theta^1$ and noise variance $\left(\sigma_\MF^1\right)^2$ are obtained by minimizing the negative log marginal likelihood in equation \eqref{equ:nlml}. Then we compute the posterior mean denoted by $\overline{{f}^{1}}(\mathbf{x}^2)$ of $\mathcal{GP}^{1}$ at $\mathbf{x}^2$ using equation \eqref{equ:pred-state}. Moreover, $\mathcal{GP}^{2}$ is devised based on the data $\left\{\left(\mathbf{x}^2, \overline{{f}^{1}}(\mathbf{x}^2)\right), \mathbf{y}^2\right\}$ and the kernel $k^{2}((\mathbf{x}, \overline{{f}^{1}}(\mathbf{x})), (\mathbf{x}', \overline{{f}^{1}}(\mathbf{x}'));\theta^2)$ defined in \eqref{equ:mfgp-kernel}. The hyperparameter $\theta^2$ and $\left(\sigma_\MF^2\right)^2$ is optimized by minimizing equation \eqref{equ:nlml}.

Before the description of the inference stage, we analyze the structure of the MFGP model and make a reasonable assumption. Let $\mathbf{x}^* \in \mathbb{R}^D$ be a predictive point. Given that the LF model $f^1(\mathbf{x}^*)$ subject to a GP prior, its posterior distribution remains to be GP, which enables the analytical computations of both the posterior mean and variance. Nevertheless, the HF model $f^2(\mathbf{x}^*)$ follows a GP prior, its posterior distribution do not follow a GP anymore due to the nonlinear mapping $f^1(\mathbf{x}^*)$, wherein the uncertainty no longer subject to a Gaussian distribution. The posterior distribution of the HF model $p\left( f^{2}\left( \mathbf{x}^*\right)\right)$ can be approximated merely by a numerical approach \cite{perdikaris2017nonlinear}, such as the Monte-Carlo integration,
\[
    \begin{aligned}
        p\left( f^{2}\left( \mathbf{x}^*\right)\right) & =p\left( f^{2}\left( \mathbf{x}^* ,f^{1}\left( \mathbf{x}^*\right)\right) \bigg| f^{1}\left( \mathbf{x}^*\right) ,\mathbf{x}^* ,\mathbf{y}^{2} ,\mathbf{x}^{2}\right)\\
        & =\int p\left( f^{2}\left( \mathbf{x}^* ,f^{1}\left( \mathbf{x}^*\right)\right) \bigg| ,\mathbf{x}^* ,\mathbf{y}^{2} ,\mathbf{x}^{2}\right) p\left(f^{1}\left( \mathbf{x}^*\right)\right)\mathrm{d} \mathbf{x}^*,
    \end{aligned}
\]
where the posterior distribution of the LF model $p\left(f^{1}\left( \mathbf{x}^*\right)\right)$ is a Gaussian distribution, but the conditional probability $p\left( f^{2}\left( \mathbf{x}^* ,f^{1}\left( \mathbf{x}^*\right)\right) \bigg| ,\mathbf{x}^* ,\mathbf{y}^{2} ,\mathbf{x}^{2}\right)$ is non-Gaussian.

However, such numerical computations result in unaffordable computational cost. To mitigate this burden, we introduce the following assumption, which allows to compute the posterior estimates of the HF model analytically. Here we assume that the posterior variance of the LF model $\overline{{f}^{1}}(\mathbf{x}^*)$ is sufficiently small to be ignored, thereby, it could be treated as a deterministic variable. Subsequently, the posterior distribution of the HF model turns into as a Gaussian distribution and the posterior mean and variance could be computed analytically. Since we are focusing on the inference of the HF model, this assumption is reasonable and efficient in practical computations. Numerical experiments illustrate the efficiency of this treatment.

During the inference stage, we initially compute the posterior mean of $\mathcal{GP}^{1}$ at predictive input $\mathbf{x}^*$ by equation \eqref{equ:pred-state} denoted by $\overline{{f}^{1}}(\mathbf{x}^*)$. We obtain the posterior mean and variance of $\mathcal{GP}^{2}$ at $\left(\mathbf{x}^*, \overline{{f}^{1}}(\mathbf{x}^*)\right)$ represented by $\overline{{f}^{2}}(\mathbf{x}^*)$, $\Var\left({f}^{2}(\mathbf{x}^*)\right)$, respectively. Furthermore, we can also employ MFGP to infer the posterior mean and variance of partial derivatives of $\mathcal{GP}^{2}$ at $\left(\mathbf{x}^*, \overline{{f}^{1}}(\mathbf{x}^*)\right)$ by equation \eqref{equ:pred-der} denoted by $(\overline{\partial {f}^{2}}(\mathbf{x}^*))_{j}$, $\Var\left((\overline{\partial {f}^{2}}(\mathbf{x}^*))_{j}\right)$. Here, the subscript $j$ means the partial derivatives to the $j$-th element of $\mathbf{x}^*$, so we tend to use the $\overline{\partial {f}^{2}}(\mathbf{x}^*)$ and $\Var\left(\overline{\partial {f}^{2}}(\mathbf{x}^*)\right)$ to abbreviate them. Details on partial derivatives of the MFGP kernel \eqref{equ:mfgp-kernel} are demonstrated in \ref{sec:app-der-MFGP}.

The MFGP algorithm for the case of bi-fidelity data ($L=2$) is summarized in Algorithm \ref{alg:MFGP}, where step 1 to 3 focus on learning the optimal hyperparameters in the kernel and the noise variance, while steps 4 to 6 are concerned with predictions for the state and its partial derivatives of the HF model.

\begin{remark}
    As for the inference of partial derivatives, we can utilize the automatic differentiation techniques to track the gradient information during the process of state prediction, such as PyTorch \cite{paszke2019pytorch}. This approach offers the advantage that we only need to conduct state prediction once, wherein the derivatives are obtained through automatic differentiation. However, it brings two drawbacks. One is that the posterior mean results of first-order partial derivatives is accurate, whereas accuracy diminishes for the second and higher-order derivatives due to the existence of noise. The other is that this method cannot offer an approximation to the posterior variance.
\end{remark}

\begin{algorithm}[ht]
	\caption{Multi-fidelity GP regression (MFGP)}
	\label{alg:MFGP}
	{\bf Input:} MF training data $\mathcal{D}^l = \curly{(\mathbf{X}^l, \mathbf{y}^l)}, l=1, 2$, predictive input $\mathbf{x}^*$, kernel functions $k^1(\mathbf{x}, \mathbf{x}';\theta^1)$, and $k^{2}((\mathbf{x}, \overline{{f}^{1}}(\mathbf{x})), (\mathbf{x}', \overline{{f}^{1}}(\mathbf{x}'));\theta^2)$.
	\begin{algorithmic}[1]
        \State Construct the $\mathcal{GP}^{1}$ based on $\mathcal{D}^1$ and $k^1(\mathbf{x}, \mathbf{x}';\theta^1)$, and learn $\theta^1$ and $\left(\sigma_\MF^1\right)^2$ by minimizing equation \eqref{equ:nlml}.
        \State Compute the posterior mean $\overline{{f}^{1}}(\mathbf{x}^2)$ of $\mathcal{GP}^{1}$ at $\mathbf{x}^2$ using equation \eqref{equ:pred-state}.
        \State Construct $\mathcal{GP}^{2}$ based on $\left\{\left(\mathbf{x}^2, \overline{{f}^{1}}(\mathbf{x}^2)\right), \mathbf{y}^2\right\}$ and $k^{2}((\mathbf{x}, \overline{{f}^{1}}(\mathbf{x})), (\mathbf{x}', \overline{{f}^{1}}(\mathbf{x}'));\theta^2)$ in equation \eqref{equ:mfgp-kernel}, and learn $\theta^2$ and $\left(\sigma_\MF^2\right)^2$ by minimizing equation \eqref{equ:nlml}.
        \State Compute the posterior mean $\overline{{f}^{1}}(\mathbf{x}^*)$ of $\mathcal{GP}^{1}$ at $\mathbf{x}^*$ using equation \eqref{equ:pred-state}.
        \State Calculate the posterior mean $\overline{{f}^{2}}(\mathbf{x}^*)$ and variance $\Var\left(\overline{{f}^{2}}(\mathbf{x}^*)\right)$ of $\mathcal{GP}^{2}$ at $\left(\mathbf{x}^*, \overline{{f}^{1}}(\mathbf{x}^*)\right)$ using equation \eqref{equ:pred-state}.
        \State Calculate the posterior mean $\overline{\partial {f}^{2}}(\mathbf{x}^*)$ and variance of partial derivative $\Var\left(\overline{\partial {f}^{2}}(\mathbf{x}^*)\right)$ of $\mathcal{GP}^{2}$ at $\left(\mathbf{x}^*, \overline{{f}^{1}}(\mathbf{x}^*)\right)$ using equation \eqref{equ:pred-der}.
	\end{algorithmic}
    {\bf Output:} $\overline{{f}^{2}}(\mathbf{x}^*)$, $\Var\left({f}^{2}(\mathbf{x}^*)\right)$, $\overline{\partial {f}^{2}}(\mathbf{x}^*)$, $\Var\left(\overline{\partial {f}^{2}}(\mathbf{x}^*)\right)$.
\end{algorithm}

\subsubsection{MFGP-SINDy}

Similar with GP-SINDy, MFGP-SINDy also consists of three steps: constructing the MFGP model based on MF data, inferring the state and its partial derivatives, and computing a sparse solution of a WLS problem.

Firstly, MFGP is exploited as a surrogate model, which enables one to sample more points in the grid. Next, we use MFGP to infer the variables of interest. The MFGP kernel provides a sophisticated and powerful approach wherein the partial derivatives are calculated though the differentiation of the kernel function. This analytical manner is highly appropriate suitable for sparse and noisy observed data. Several local surrogate approaches can smooth the noisy data, but their parameters are difficult to select. Since the outcomes heavily depend on the choice of parameters, particularly when the noisy data is scarce. GP and MFGP have the merit that they do not need to select the parameters manually.

Finally, we obtain the discovered system via solving the WLS problem, which is accomplished by STWLS (Algorithm \ref{alg:STWLS}). Combined with GP-SINDy (Algorithm \ref{alg:GPSINDy}) and MFGP (Algorithm \ref{alg:MFGP}), the Multi-Fidelity Gaussian Process based SINDy (MFGP-SINDy) algorithm for sparse and noisy MF observed data is outlined in Algorithm \ref{alg:MFGPSINDy}. The number of HF observed data is less than that of LF due to fact that the cost of numerical simulation of HF model is more expensive. The size of HF data is a significant factor for the trade-off between computational cost and accuracy, which will be demonstrated in Section \ref{sec:Numerical}.

\begin{algorithm}[H] 
	\caption{Multi-Fidelity Gaussian Process based SINDy (MFGP-SINDy)}
	\label{alg:MFGPSINDy}
	{\bf Input:} MF data $\mathcal{D}^l =\left\{\mathbf{  X}^l \in \mathbb{R}^{D\times N_l}, \mathbf{U}^l \in \mathbb{R}^{N_l\times d}\right\}$ with $l=1, 2$, test input $\mathbf{X}_p \in\mathbb{R}^{D\times N'}$, parameters of STWLS $\Lambda$, $\eta$, $K$, $J$.
	\begin{algorithmic}[1]
        \For{$i \in \{ 1, \cdots, d\}$}
            \State Construct the MFGP model using training data $\{\mathbf{X}^l, \mathbf{U}_i^l\}$ with $l=1, 2$ and the kernel functions $k^1$, $k^2$, and learn the optimal hyperparameters $\theta_i^1$, $\theta_i^2$, $\left(\sigma_\MF^1\right)^{2}_i$ and $\left(\sigma_\MF^2\right)^{2}_i$ using steps 1 to 3 in Algorithm \ref{alg:MFGP}.
            \State Infer the posterior $\overline{\mathbf{U}}_i$, $\overline{\dot{\mathbf{U}}}_i$, $\Var\left(\dot{\mathbf{U}_i}\right)$, and $\overline{\mathbf{V}}_i$ at test input $\mathbf{X}_p$ using steps 4 to 6 in Algorithm \ref{alg:MFGP}.
            \State Compute the weight matrix $\mathbf{W}_i = \left(\diag\left(\Var(\dot{\mathbf{U}_i})\right)\right)^{-1}$.
        \EndFor
        \State Assemble the data matrices, $\overline{\mathbf{U}} = [\overline{\mathbf{U}_1}, \overline{\mathbf{U}_2}, \cdots, \overline{\mathbf{U}_d}]$, $\overline{\mathbf{V}} = [\overline{\mathbf{V}_1}, \overline{\mathbf{V}_2}, \cdots, \overline{\mathbf{V}_d}]$.
        \State Compute the function library matrix $\boldsymbol{\Phi}  = \boldsymbol{\Phi} (\overline{\mathbf{U}}, \overline{\mathbf{V}})$.
        \For{$i \in \{ 1, \cdots, d\}$}
            \State $\mathbf{C}[:, i]$ = STWLS($\boldsymbol{\Phi}$, $\overline{\dot{\mathbf{U}}}_i$, $\mathbf{W}_{i}$, $\Lambda$, $\eta$, $K$, $J$).
        \EndFor
	\end{algorithmic}
    {\bf Output:} Sparse coefficient matrix $\mathbf{C}$.
\end{algorithm}

\section{Numerical experiments}
\label{sec:Numerical}
In this section, we conduct three numerical experiments to demonstrate the performance of GP-SINDy and MFGP-SINDy, including the Lorenz system, the Burgers' equation, and the KdV equation. In the discovery of Lorenz system, we focus on the performance of GP-SINDy in handling SF observed data. The subsequent part involves the application of MFGP-SINDy to MF data scenarios, which contains the discovery of the Burgers' equation and the KdV equation. Our approach is compared with other alternative methods in the context of MF data. Meanwhile, the effect of training data sizes on the outcomes is explored.

In GP-SINDy or MFGP-SINDy, all hyperparameters in the kernel and the noise variance are initialized randomly, and they are learned by minimizing equation \eqref{equ:nlml}. This is implemented through the \texttt{Rprop} optimizer in PyTorch \cite{paszke2019pytorch}. For the parameters in STWLS, we set $K=20$, $J=10$ for all experiments.

Prior to conducting the experiments, we outline some experimental setups. The noise-free data of ODEs is generated utilizing the \texttt{ode45} function within MATLAB that employs absolute and relative tolerance set of $10^{-8}$. In the case of PDEs, the noise-free data of HF model is generated using the \texttt{spin} class from the \texttt{Chebfun} library \cite{driscoll2014ChebfunGuide, robertstephany2023pdelearn}. The observed data is obtained from the clean data through adding i.i.d. Gaussian white noise with zero mean and variance $\sigma^2$, where $\sigma$ is determined by the noise ratio $\sigma_\NR$ and the clean data matrix $\mathbf{U}^*$,
\[
    \sigma = \sigma_\NR \frac{\| \mathbf{U}^* \| _{F}}{\sqrt{Nd}}.
\]
Here the positive parameter $\sigma_\NR$ is manually specified, and $\| \mathbf{U}^* \| _{F}$ represents the Frobenius norm of the matrix $\mathbf{U}^*\in\mathbb{R}^{N\times d}$. This formulation indicates that the noise-to-signal ratio is approximately equivalent to $\sigma$ \cite{messenger2021weak}, \ie, $\sigma \approx \| \mathbf{U} - \mathbf{U}^* \|_{F} / \| \mathbf{U}^* \|_{F}$.

Let $\mathbf{C}^{*}\in \mathbb{R}^{N_f\times d}$ denote the true sparse coefficient matrix, we use three error metrics to evaluate the accuracy of $\mathbf{C}$ obtained by GP-SINDy and MFGP-SINDy algorithms. The first is the maximum error of the true non-zero coefficients,
\[
    E_\infty(\mathbf{C}) =\max_{\mathbf{C}_{i, j}^{*}\neq 0}\frac{| \mathbf{C}_{i, j} -\mathbf{C}_{i, j}^{*} |}{| \mathbf{C}_{i, j}^{*} |}.
\]
The second is the relative $\ell_2$ error between the discovered coefficients $\mathbf{C}$ and the true coefficients $\mathbf{C}^{*}$,
\[
    E_2(\mathbf{C}) =\frac{\| \mathbf{C} -\mathbf{C}^{*} \| _{F}}{\| \mathbf{C}^{*} \| _{F}},
\]
And the last metric is the true positivity ratio ($TPR$) \cite{lagergren2020learning},
\[
    TPR(\mathbf{C}) = \frac{TP}{TP + FN + FP},
\]
where $TP$ represents the number of correctly discovered non-zero coefficients, $FN$ represents the number of coefficients falsely discovered as zero, and $FP$ represents the number of coefficients falsely discovered as non-zero. $TPR$ measures the percentage of correctly identified function terms, and a $TPR$ value of 1 indicates a perfect discovery of non-zero function terms.

\subsection{Lorenz system}
\label{sec:Lorenz}

The Lorenz system is described as,
\begin{equation*}
    \begin{cases}
        \dot{x} =\beta _{1} (y-x),\\
        \dot{y} =x( \beta _{2} -z) -y,\\
        \dot{z} =xy-\beta _{3} z.
    \end{cases}
\end{equation*}
Here we choose $\beta_1 = 10, \beta_2 = 28, \beta_3 = -\frac{8}{3}$, and the time domain $t\in [0, 10]$, the initial value $[x_0, y_0, z_0] = [-8, 7, 27]$. The training data is generated with time-step $\Delta_t =0.001$. During the process of sparse learning, polynomials up to the third-order are utilized as the function library, \ie, $\boldsymbol{\Phi} (u )$ = $\boldsymbol{\Phi} (x, y, z)$ = $[$$1$, $x$, $y$, $z$, $x^2$, $xy$, $xz$, $y^2$, $yz$, $z^2$, $x^3$, $x^2y$, $x^2z$, $xy^2$, $xyz$, $xz^2$, $y^3$, $y^2z$, $yz^2$, $z^3$$]$. For GP-SINDy, the predictive input is the same with the input in the training data. SE kernel is adopted in this experiment. For the parameters in STWLS, we set $\Lambda$ = $\{10^{i}$ with $i=0, -1, -2, \cdots, -5\}$ and $\eta = 5$.

We plot the ground truth and noisy observed data with a noise ratio $\sigma_\NR=0.1$ (the exact standard deviation $\sigma = 1.59742$) in Figure \ref{fig:lorenz-butter}. Our GP-SINDy algorithm is applied to these observed data, where observed data is subsampled evenly with size of 626 during the training of hyperparameters in order to mitigate the computation afford of GPR. The results of hyperparameters and noise variance in GPR are shown in Table \ref{tab:lorenz-hyperparameters-GP}. The corresponding states prediction and derivatives prediction results with uncertainty approximation obtained from GPR are exhibited in Figure \ref{fig:lorenz-state} and \ref{fig:lorenz-derivative}, respectively. We observe that the posterior variance of the derivatives data are greater near the boundary points compared to the interior points, which attributes to the availability of information on only one side rather than both sides.

\begin{figure}[H] 
    \centering
    \includegraphics[width=0.3\textwidth]{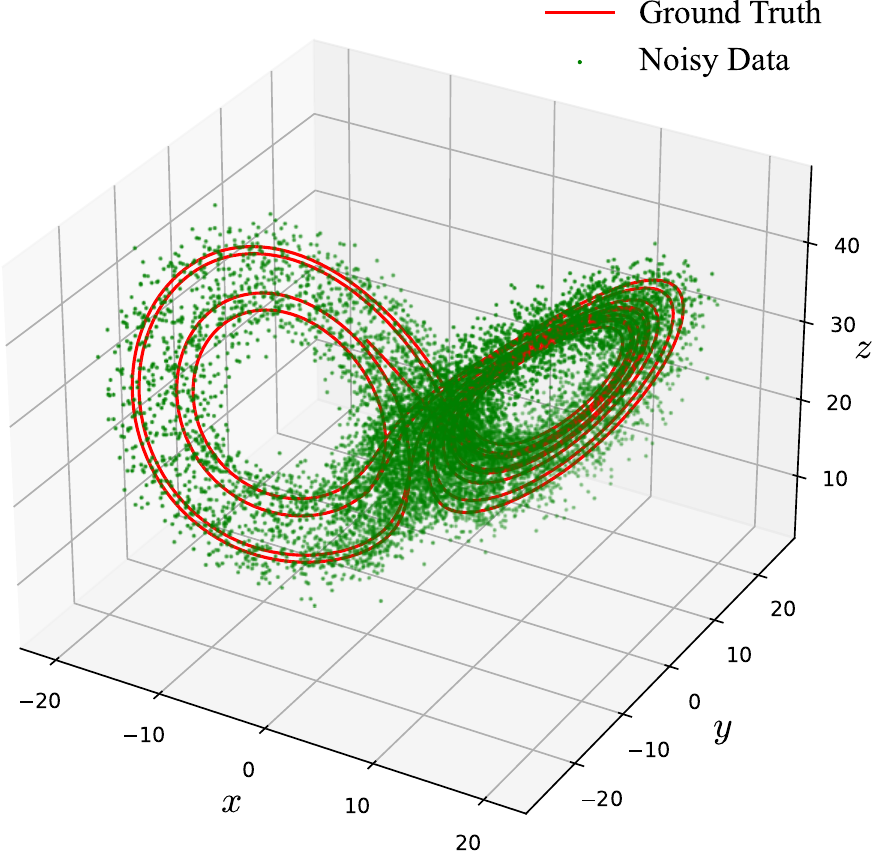}
    \caption{The ground truth and noisy data with $\sigma_\NR=0.1$.}
    \label{fig:lorenz-butter}
\end{figure}

\begin{table}[htbp]
    \centering
    \captionsetup{width=0.4\textwidth}
    \caption{Optimal hyperparameters results of GPR in GP-SINDy. Here, $x$, $y$, and $z$ are state variables, which correspond to three GP models.}
    \label{tab:lorenz-hyperparameters-GP}
    \begin{tabular}{lccc}
        \hline
         & $(\theta_i)_0$ & $(\theta_i)_1$ & $\sigma _{i}^{2}$ \\ \hline
        $x (i = 1)$ & $0.826$ & $0.023$ & $0.047$ \\ \hline
        $y (i = 2)$ & $0.865$ & $0.014$ & $0.034$ \\ \hline
        $z (i = 3)$ & $1.049$ & $0.015$ & $0.034$ \\
        \hline
    \end{tabular}
\end{table}

\begin{figure}[ht] 
    \centering
    \includegraphics[width=0.8\textwidth]{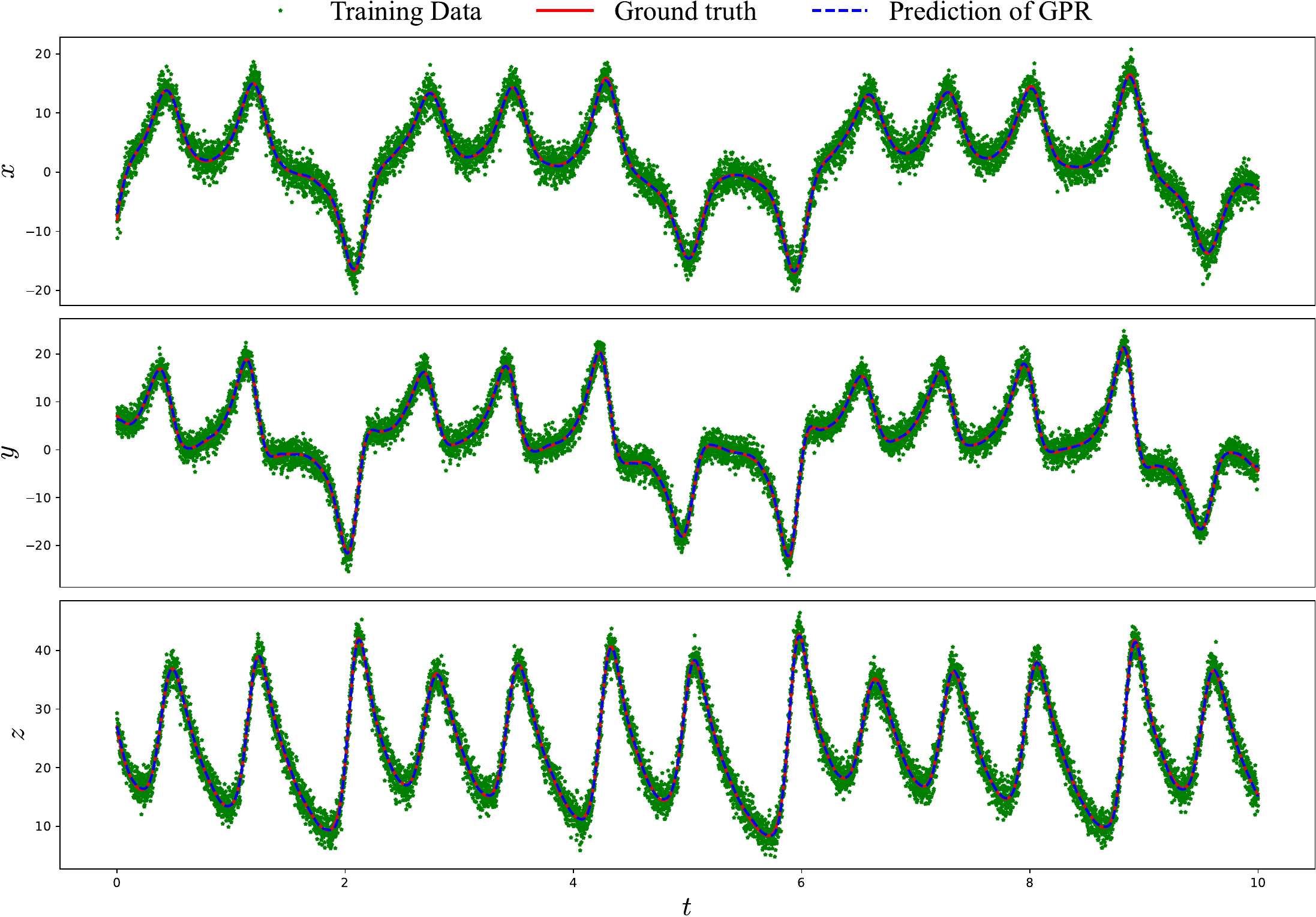}
    \caption{Comparison between the ground truth, training data, and states prediction obtained from GPR.}
    \label{fig:lorenz-state}
\end{figure}

\begin{figure}[H] 
    \centering
    \includegraphics[width=0.8\textwidth]{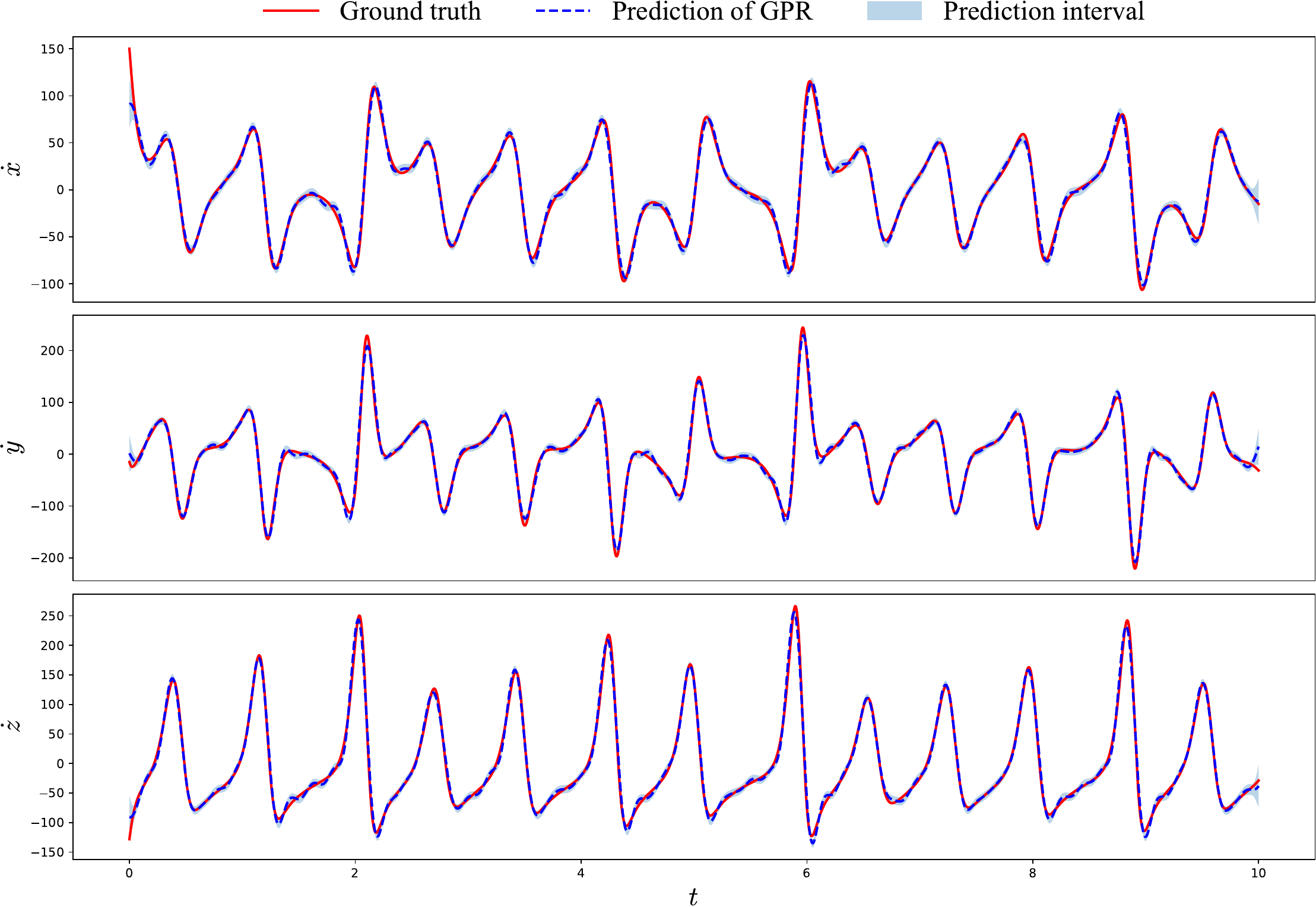}
    \caption{The prediction and uncertainty approximation of derivatives.}
    \label{fig:lorenz-derivative}
\end{figure}

Next, we compare GP-SINDy with SINDy \cite{brunton2016discovering}, Savitzky-Golay SINDy (SG-SINDy) \cite{Kaptanoglu2022}, and Modified SINDy (M-SINDy) \cite{kaheman2022automatic}, wherein the same function library is exploited. In SG-SINDy, the Savitzky-Golay filter is exploited to fit a series of local polynomials for the training data, and the derivatives are calculated by the finite difference method. Here, we set the order of local polynomials to 3 and the length of window to 10. Besides, M-SINDy is another robust algorithm that combines the automatic differentiation and time-stepping constraints. The discovered results and coefficients error of these algorithms under $\sigma_\NR=0.05, 0.1$ are summarized in Table \ref{tab:lorenz-2} and the best results are listed in bold. The comparison of discovered dynamic systems are shown in Figure \ref{fig:lorenz_gpsindy-cmp}. As for a local fitting model, the performance of SG-SINDy closely depends on the selection of parameters (the length of window). M-SINDy requires to solve an ill-posed optimization problem, since the degree of freedom of the optimized variables is more than the known data, and we observe that it is prone to overfit and fall into a local minimal. As pointed out by \cite{kaheman2022automatic}, it may produce a system that has a similar behavior to the actual system in given time interval, but the symbolic repression is extremely different from the exact solution.

\begin{table}[H] 
    \centering
    \caption{Discovered results and error under two noise levels ($\sigma_\NR \in \{0.05, 0.1\}$), where excessive results in SINDy are replaced by ``$\cdots$''.}
    \label{tab:lorenz-2}
    \footnotesize
    \begin{tabular}{p{1.25cm}|lll|p{0.7cm}|p{0.7cm}|p{0.7cm}}
        \hline
        \textbf{Method} & \multicolumn{3}{c|}{\textbf{Discovered equation} ($\sigma_\NR=0.05$)} & $E_\infty(\%)$ & $E_2(\%)$ &$TPR$\\ \hline
        SINDy   & $\begin{aligned}
                    \dot{x}= & -0.814 + 4.078y - 0.265xz\\
                             & + 0.149yz
                    \end{aligned}$
                & $\begin{aligned}
                    \dot{y}= & -1.249 + 28.651x - 1.276y\\
                    & + 0.131z - 0.113x^2 + \cdots
                    \end{aligned}$
                & $\begin{aligned}
                    \dot{z}= & -1.761 - 0.358x + 0.3y\\
                    & - 2.58z + 0.998xy
                    \end{aligned}$
                & 100.00 & 37.70 & 0.33 \\ \hline
        SG-SINDy    & $\dot{x} =0.284 - 9.83x + 9.746y$
                    & $\dot{y} =27.853x - 1.273y - 0.988xz$
                    & $\dot{z} =3.312 - 2.778z + 0.986xy$
                    & 27.34 & 10.64 & 0.78 \\ \hline
        M-SINDy     &$\dot{x} =-9.908x + 9.914y$
                    &$\dot{y} =28.083x - 1.033y - 1.002xz$
                    &$\dot{z} =-0.396 - 2.654z + 1.002xy$
                    & 3.30 & 1.35 & 0.88 \\ \hline
        GP-SINDy    & $\dot{x} =-9.944x + 9.933y$
                    & $\dot{y} =27.916x - 0.991y - 0.999xz$
                    & $\dot{z} =-2.663z + 0.999xy$
                    & \textbf{0.93} & \textbf{0.39} & \textbf{1} \\
        \hline
        \hline
        \textbf{Method} & \multicolumn{3}{c|}{\textbf{Discovered equation} ($\sigma_\NR=0.1$)} & $E_\infty(\%)$ & $E_2(\%)$ &$TPR$\\ \hline
        SINDy   & $\begin{aligned}
                        \dot{x}= & 56.124 + 9.489x - 0.815y\\
                                & - 5.446z - 0.502xz + \cdots
                    \end{aligned}$
                & $\begin{aligned}
                        \dot{y}= & 24.768x + 1.948y + 0.12z\\
                                & - 0.151x^2 + 0.345xy + \cdots
                    \end{aligned}$
                & $\begin{aligned}
                        \dot{z}= & -4.71 - 1.04x + 0.79y\\
                                & - 2.437z + 1.001xy
                    \end{aligned}$
                & 294.76 & 193.44 & 0.35 \\ \hline
        SG-SINDy    & $\dot{x} =0.82 - 9.429x + 9.156y$
                    & $\begin{aligned}
                            \dot{y}= & 2.955 + 26.468x - 1.609y\\
                                    & - 0.114z - 0.935xz
                        \end{aligned}$
                    & $\begin{aligned}
                            \dot{z}= & 0.176 + 0.28y - 2.568z\\
                                    & 0.687xy + 0.187y^2
                        \end{aligned}$
                    & 60.94 & 11.62 & 0.54 \\ \hline
        M-SINDy     & $\dot{x} =-9.747x + 9.769y$
                    & $\dot{y} =28.165x - 1.062y - 1.004xz$
                    & $\dot{z} =-0.965 - 2.636z + 1.007xy$
                    & 6.15 & 3.30 & 0.88 \\ \hline
        GP-SINDy    & $\dot{x} =-9.881x + 9.861y$
                    & $\dot{y} =27.721x - 0.942y - 0.995xz$
                    & $\dot{z} =-2.659z + 0.998xy$
                    & \textbf{5.81} & \textbf{1.08} & \textbf{1} \\
        \hline
    \end{tabular}
\end{table}

\begin{figure}[H] 
    \centering
    \includegraphics[width = 0.9\textwidth]{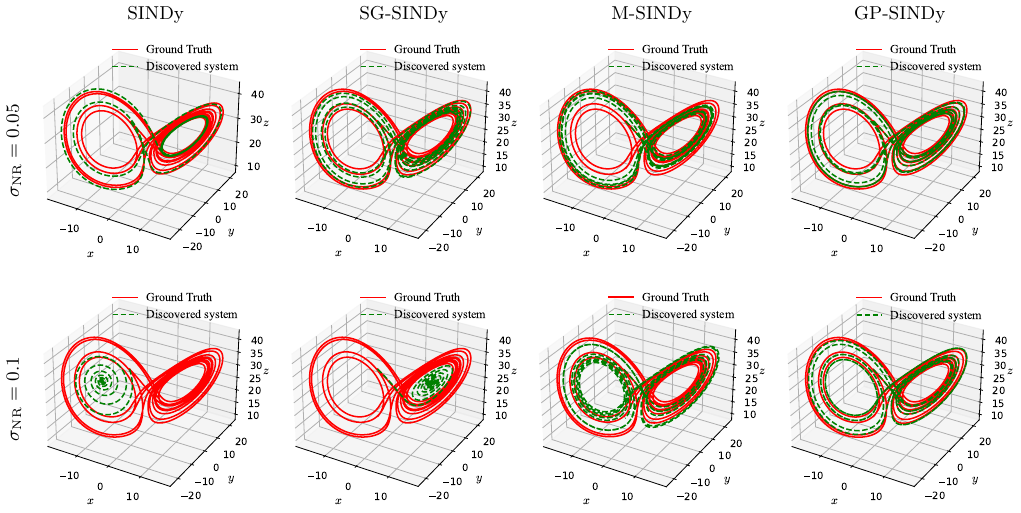}
    \caption{The visualization of discovered dynamic systems, and these figures correspond to Table \ref{tab:lorenz-2}, where all initial conditions are set to [-8, 7, 27].}
    \label{fig:lorenz_gpsindy-cmp}
\end{figure}

\subsection{Burgers' equation}
\label{sec:Burgers}

Consider the 1D Burgers' equation,
\begin{equation}\label{equ:burgers}
    u_{t} = \nu u_{xx} - uu_{x},
\end{equation}
with the initial condition $u(x, 0) = -\sin(\pi x /8)$ and the periodic boundary condition $u(-8, t) = u(8, t)$, where $t\in[0, 10]$, $x\in[-8, 8]$. Here, the diffusion coefficient $\nu=0.5$.

\subsubsection{GP-SINDy for SF data}
\label{sec:Burgers-GP}
The SF data is generated by the fine spatiotemporal grid using the \texttt{spin} class from the \texttt{Chebfun} library. The time step size is $\Delta {t} = 0.002$ and the spatial grid size $\Delta {x} = 0.00625$. The training data is obtained by adding noise and downsampling of the clean data. Here, we downsample the data evenly with size $n_s$, which is referred to the size of SF data. For the sake of notation simplicity, we describe the training data using its size, instead of the downsampling step. For instance, $n_s$ $(41 \times 65 = 2665)$ represents a uniform spatiotemporal grid with time step 41 and space size 65. The function library contains spatial partial derivatives up to the second-order and polynomials up to the second-order, \ie, $\boldsymbol{\Phi} (u ) = [1, u, u_x, u_{xx}, u^2, uu_x, uu_{xx}, u_{x}^2, u_xu_{xx}, u_{xx}^2]$. The predictive points are set as the spatiotemporal grid in $ [0, 10] \cup [-8, 8]$ with $\Delta {t}' = 0.1$, $\Delta {x}' = 0.0625$ throughout this experiment. Besides, we set $\Lambda$ = $\{i/2$ with $i=0, 1, \cdots, 10\}$ and $\eta = 150$ for Algorithm \ref{alg:GPSINDy}.

First, we demonstrate the performance of GP-SINDy with SF data, here we use three SF data sizes, $n_{s1}$ $(41 \times 65)$, $n_{s2}$ $(41 \times 33)$, and $n_{s3}$ $(41 \times 17)$. We show the identification results and error under various noise levels in Table \ref{tab:Burgers-GP-uniform-symbol}, where $\sigma$ represents the exact standard deviation of noise. We observe that for each SF data size $n_s$, the results deteriorate as the noise level get larger. Results in Table \ref{tab:Burgers-GP-uniform-symbol} show that GP-SINDy is capable of discovering the correct differential equation when the observation data is sparse and noisy. In the case of noise level $\sigma_\NR=0.3$, although the corresponding coefficient errors $E_\infty$ and $E_2$ are too large to accept, GP-SINDy could discover the correct function terms, \ie,$TPR$ is 1. Additionally, as the SF training data gets coarser from $n_{s1}$ to $n_{s2}$ to $n_{s3}$, the results get worse consistently at each specific noise level. The optimal hyperparameters of GPR are obtained by minimizing the negative log marginal likelihood, and we list the optimal values under different noise levels in Table \ref{tab:Burgers-hyperparameters-GP}.

\begin{table}[H] 
    \centering
    \caption{Identification results and error of GP-SINDy with three SF data sizes.}
    \label{tab:Burgers-GP-uniform-symbol}
    \begin{tabular}{c|l|l|l|l|l|l}
        \hline
        \textbf{SF data size} & $\sigma_\NR$ & $\sigma$ & \textbf{Discovered equation} & $E_\infty$ (\%) & $E_2 (\%)$ & $TPR$\\
        \hline
        \multirow{4}*{$n_{s1}$}
            & 0.02 & 0.009 & $u_t = 0.495 u_{xx} - 0.983 uu_x$ & 1.69 & 1.57 & 1 \\ \cline{2-7}
            & 0.1 & 0.046 & $u_t = 0.479 u_{xx} - 0.933 uu_x$ & 6.73 & 6.30 & 1 \\ \cline{2-7}
            & 0.2 & 0.093 & $u_t = 0.469 u_{xx} - 0.87 uu_x$ & 12.96 & 11.92 & 1 \\ \cline{2-7}
            & 0.3 & 0.139 & $u_t = 0.462 u_{xx} - 0.822 uu_x$ & 17.83 & 16.30 & 1 \\ \hline
        \addlinespace
        \hline
        \multirow{4}*{$n_{s2}$}
            & 0.02 & 0.009 & $u_t = 0.492 u_{xx} - 0.969 uu_x$ & 3.08 & 2.85 & 1 \\ \cline{2-7}
            & 0.1 & 0.046 & $u_t = 0.467 u_{xx} - 0.89 uu_x$  & 11.02 & 10.29 & 1 \\ \cline{2-7}
            & 0.2 & 0.093 & $u_t = 0.445 u_{xx} - 0.826 uu_x$ & 17.40 & 16.33 & 1 \\ \cline{2-7}
            & 0.3 & 0.139 & $u_t = 0.426 u_{xx} - 0.775 uu_x$ & 22.48 & 21.17 & 1 \\ \hline
        \addlinespace
        \hline
        \multirow{4}*{$n_{s3}$}
            & 0.02 & 0.009 & $u_t = 0.488 u_{xx} - 0.945 uu_x$ & 5.51 & 5.03 & 1 \\ \cline{2-7}
            & 0.1 & 0.046 & $u_t = 0.457 u_{xx} - 0.843 uu_x$ & 15.72 & 14.58 & 1 \\ \cline{2-7}
            & 0.2 & 0.093 & $u_t = 0.446 u_{xx} - 0.774 uu_x$ & 22.61 & 20.79 & 1 \\ \cline{2-7}
            & 0.3 & 0.139 & $u_t = 0.433 u_{xx} - 0.694 uu_x$ & 30.55 & 27.98 & 1 \\ \hline
    \end{tabular}
\end{table}

\begin{table}[H]
    \centering
    \captionsetup{width=0.4\textwidth}
    \caption{Optimal hyperparameter results of GPR for data size $n_{s1}$ under various noise levels, where $i=1$.}
    \label{tab:Burgers-hyperparameters-GP}
    \begin{tabular}{llcccc}
        \hline
        $\sigma_\NR$ & $\sigma$ & $(\theta_i)_0$ & $(\theta_i)_1$ & $(\theta_i)_2$ & $\sigma _{i}^{2}$ \\ \hline
        $0.02$ & 0.009 & $0.493$ & $10.112$ & $2.542$ & $0.0004$ \\ \hline
        $0.1$ & 0.046 & $0.657$ & $19.651$ & $2.909$ & $0.010$ \\ \hline
        $0.2$ & 0.093 & $0.617$ & $21.896$ & $3.337$ & $0.039$ \\ \hline
        $0.3$ & 0.139 & $0.597$ & $23.335$ & $3.599$ & $0.084$ \\
        \hline
    \end{tabular}
\end{table}

Next, we compare GP-SINDy with other methods, including Tik in PDE-FIND \cite{rudy2017datadriven}, Robust IDENT (rIDENT) \cite{he2022robust, he2023group}, and SG-Tik \cite{rudy2017datadriven}. For Tik, the Tikhonov differentiation \cite{rudy2017datadriven} is used for the derivatives computations. In rIDENT, we select the Savitzky-Golay filter as the smoother during the process of successively denoised differentiation (SDD) and subspace pursuit cross-validation (SC) for the sparse discovery of the differential equations. SG-Tik is the combination of Tik and rIDENT, and we use the Savitzky-Golay filter to smooth the noisy data and apply Tik to the smoothed data for the derivatives computations.

For Tik, rIDENT, and SG-Tik, the data used for sparse identification only focuses on the training data grid. However, GP-SINDy provides an alternative, which allows to infer the data on a finer grid and apply to discover the differential equations. To ensure the fairness and illustrate the smoothing effect instead of the interpolation, we add a control method called GP-SINDy* based on GP-SINDy, where these two methods differ only on the prediction points. Specifically, GP-SINDy uses the above default prediction points ($\Delta {t}' = 0.1$, $\Delta {x}' = 0.0625$), and GP-SINDy* infers the data only in the training data points. For sparse training data, the size of predictive points in GP-SINDy is larger than that of GP-SINDy*.

The error comparison ($E_2$) with three SF data sizes $n_{s3}$ $(41 \times 81)$, $n_{s4} $ $(41 \times 129)$, and $n_{s5} $ $(41 \times 257)$ are listed in Table \ref{tab:Burgers-GP-cmp-diffsizes}. We can see that GP-SINDy* outperforms Tik, rIDENT, and SG-Tik under these noise levels, and GP-SINDy gives slightly better results than GP-SINDy*. Results demonstrate that GP-SINDy has better denoising effect than other three methods, and it is capable to discover the correct differential equations with sparse and noisy data. It is capable to infer the states or derivatives analytically at any point in the input domain with their uncertainty. Tik can discover the correct function terms when the noise level is low, but fails when the noise level is large since the data is sparse and highly corrupted. Notably, when the noise level is relatively high ($\sigma_\NR=0.2$), increasing the size of the dataset does not necessarily lead to more accurate results. In other words, when the noise level is large, using less training data may yield better outcomes. In rIDENT and SG-Tik, the smoothing parameters of the Savitzky-Golay filter (windows parameters) are difficult to choose, and they are closely related to the data sizes, which is a common issue for the local parameterized surrogate model. For different sizes of data, we need to adjust them manually. However, GP-SINDy offers an elegant manner that constructs a non-parametric model, which significantly increases the applicability.

\begin{table}[H] 
    \centering
    \captionsetup{width=0.8\textwidth}
    \caption{Identification error $E_2 (\%)$ of Tik, rIDENT, SG-Tik, GP-SINDy* and GP-SINDy with three SF data sizes, and we use ``---'' to represent the failed discovery and give the discovered function terms following.}
    \label{tab:Burgers-GP-cmp-diffsizes}
    \begin{tabular}{c|c|l|l|l|l|l}
        \hline
        $\sigma_\NR$ & \textbf{Size} & Tik (\%) & rIDENT (\%) & SG-Tik (\%) & GP-SINDy* (\%) & GP-SINDy (\%) \\
        \hline
        \multirow{3}*{0.02}
            & $n_{s3}$ & 5.36 & 5.82 & 4.62 & 1.27 & \textbf{1.24} \\ \cline{2-7}
            & $n_{s4}$ & 5.09 & 5.75 & 3.87 & 0.90 & \textbf{0.89} \\ \cline{2-7}
            & $n_{s5}$ & 4.73 & 5.64 & 2.84 & 0.78 & \textbf{0.78} \\ \hline
            \addlinespace
            \hline
            \multirow{3}*{0.1}
            & $n_{s3}$ & --- ($uu_x, u_xu_{xx}$) & 13.37 & 9.87 & 6.97 & \textbf{6.89} \\ \cline{2-7}
            & $n_{s4}$ & --- ($uu_x$) & 12.61 & 9.12 & 4.89 & \textbf{4.83} \\ \cline{2-7}
            & $n_{s5}$ & --- ($uu_x$) & 11.66 & 7.32 & 3.87 & \textbf{3.84} \\ \hline
            \addlinespace
            \hline
            \multirow{3}*{0.2}
            & $n_{s3}$ & --- ($uu_x$) & 16.50 & 10.91 & 10.68 & \textbf{10.52} \\ \cline{2-7}
            & $n_{s4}$ & --- ($uu_x$) & 21.46 & 18.38 & 11.35 & \textbf{11.26} \\ \cline{2-7}
            & $n_{s5}$ & --- ($u, uu_x, u_xu_{xx}$) & 21.27 & 15.86 & 10.16 & \textbf{10.10} \\ \hline
    \end{tabular}
\end{table}

\subsubsection{MFGP-SINDy for MF data}
\label{sec:Burgers-MFGP}

In this part, we test the performance of MFGP-SINDy with MF data. The HF data is generated by the fine spatiotemporal grid using the \texttt{spin} class from the \texttt{Chebfun} library. The time step size is $\Delta {t} = 0.002$ and the spatial grid size $\Delta {x} = 0.00625$. The LF data is generated by the finite different method with coarse spatiotemporal grid $\Delta {t}_l = 0.25$, $\Delta {x}_l = 0.5$. We interpolate the LF data linearly on the fine grid to generate more LF data. The HF ground truth of equation \eqref{equ:burgers} and LF data before interpolation are depicted in Figure \ref{fig:burgers-ground_truth}. In MFGP-SINDy, the LF and HF training data are corrupted by the noise with the same noise level $\sigma_\NR$. We use $n_h$ and $n_l$ to denote the size of HF and LF data, respectively. Three SE kernels $k_\rho, k_f, k_\delta$ are utilized in the MFGP kernel \eqref{equ:mfgp-kernel}.

\begin{figure}[H] 
    \centering
    \subfloat{\includegraphics[width = 0.34\textwidth]{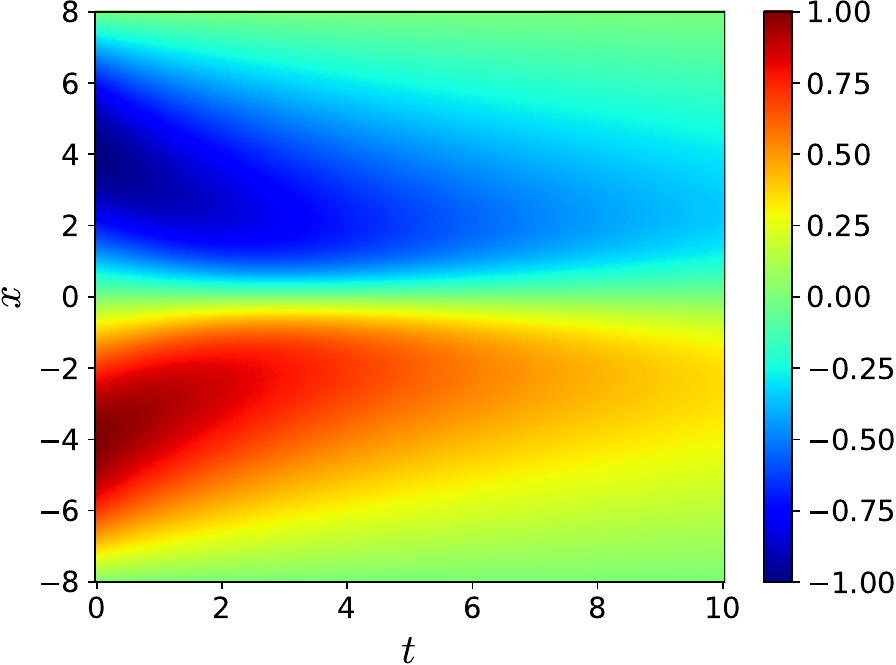}}
    \hspace{0.5cm}
    \subfloat{\includegraphics[width = 0.34\textwidth]{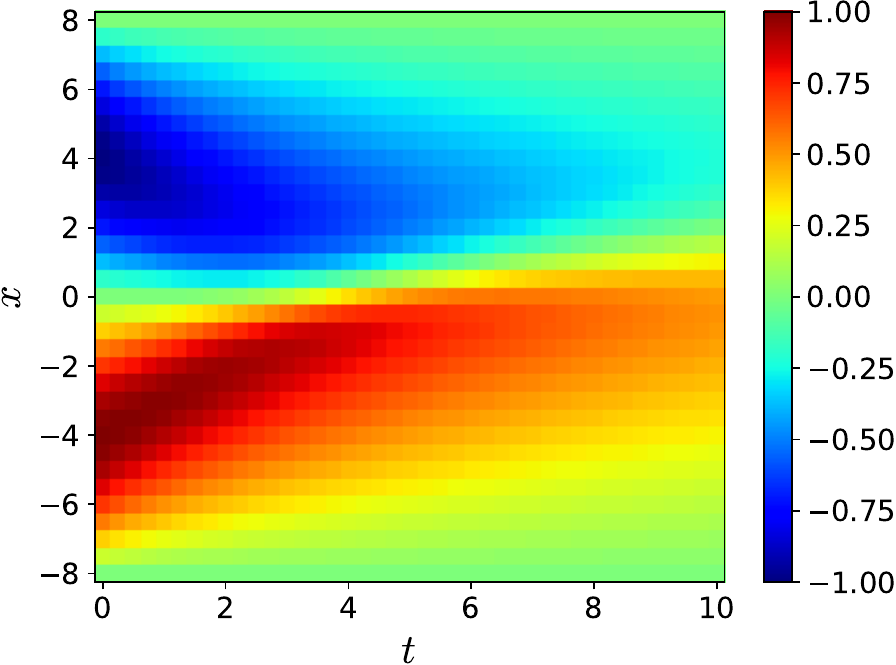}}
    \caption{The HF data and LF data (before interpolation) in Burgers' equation \eqref{equ:burgers}.}
    \label{fig:burgers-ground_truth}
\end{figure}

We consider three evenly sampled HF training data sizes $n_{h1}=n_{s1} (41 \times 65)$, $n_{h2}=n_{s2} (41 \times 33)$, $n_{h3} = n_{s3} (41 \times 17)$ and LF training data size $n_{l1}$ $(51 \times 81)$. Meanwhile, we investigate the result of GP-SINDy with data size $n_{l1}$ as a comparison experiment, and the results are listed in Table \ref{tab:Burgers-GP-onlyLF-symbol}. We can see that all results under these noise levels are not satisfactory. The discovered outcomes of MFGP-SINDy with three MF data sizes $(n_{h1}, n_{l1})$, $(n_{h2}, n_{l1})$, and $(n_{h3}, n_{l1})$ are listed in Table \ref{tab:Burgers-MFGP-uniform-symbol}. Compared with the results in Table \ref{tab:Burgers-GP-uniform-symbol}, the LF data improves the accuracy of outcomes dramatically, especially when the noise levels are 0.2, 0.3. This indicates the robustness of MFGP-SINDy to large noise levels.

\begin{table}[H] 
    \centering
    \caption{Identification results and error of GP-SINDy with data size $n_{l1}$.}
    \label{tab:Burgers-GP-onlyLF-symbol}
    \begin{tabular}{l|l|l|l|l}
        \hline
        $\sigma_\NR$ & \textbf{Discovered equation} & $E_\infty$ (\%) & $E_2 (\%)$ & $TPR$\\
        \hline
            0.02 & $u_t = 0.478 u_{xx} - 1.007 uu_x - 0.246u u_{xx}$ & 4.45 & 22.13 & 0.67 \\ \cline{1-5}
            0.1 & $u_t = 0.44 u_{xx} - 0.978 uu_x + 0.231 u_{xx}^2$ & 12.06 & 21.44 & 0.67 \\ \cline{1-5}
            0.2 & $u_t = 0.307 u_{xx} - 1.204 uu_x - 0.692u_x u_{xx}$ & 38.53 & 66.76 & 0.67 \\ \cline{1-5}
            0.3 & $u_t = -1.224 uu_x - 1.572u_x u_{xx}$ & 100.00 & 148.87 & 0.33 \\ \hline
    \end{tabular}
\end{table}

\begin{table}[H] 
    \centering
    \caption{Identification results and error of MFGP-SINDy with three MF data sizes.}
    \label{tab:Burgers-MFGP-uniform-symbol}
    \begin{tabular}{c|l|l|l|l|l}
        \hline
        \textbf{MF data size} & $\sigma_\NR$ & \textbf{Discovered equation} & $E_\infty$ (\%) & $E_2 (\%)$ & $TPR$ \\
        \hline
        \multirow{4}*{$(n_{h1}, n_{l1})$}
            & 0.02 & $u_t = 0.495 u_{xx} - 0.995 uu_x$ & 1.04 & 0.67 & 1 \\ \cline{2-6}
            & 0.1 & $u_t = 0.475 u_{xx} - 0.966 uu_x$ & 4.95 & 3.75 & 1 \\ \cline{2-6}
            & 0.2 & $u_t = 0.465 u_{xx} - 0.933 uu_x$ & 7.05 & 6.75 & 1 \\ \cline{2-6}
            & 0.3 & $u_t = 0.461 u_{xx} - 0.912 uu_x$ & 8.81 & 8.63 & 1 \\ \hline
        \addlinespace
        \hline
        \multirow{4}*{$(n_{h2}, n_{l1})$}
            & 0.02 & $u_t = 0.493 u_{xx} - 0.99 uu_x$ & 1.46 & 1.14 & 1 \\ \cline{2-6}
            & 0.1 & $u_t = 0.458 u_{xx} - 0.903 uu_x$ & 9.68 & 9.45 & 1 \\ \cline{2-6}
            & 0.2 & $u_t = 0.456 u_{xx} - 0.89 uu_x$ & 11.04 & 10.62 & 1 \\ \cline{2-6}
            & 0.3 & $u_t = 0.465 u_{xx} - 0.863 uu_x$ & 13.70 & 12.65 & 1 \\ \hline
        \addlinespace
        \hline
        \multirow{4}*{$(n_{h3}, n_{l1})$}
            & 0.02 & $u_t = 0.478 u_{xx} - 0.98 uu_x$ & 4.33 & 2.61 & 1 \\ \cline{2-6}
            & 0.1 & $u_t = 0.46 u_{xx} - 0.9 uu_x$ & 10.04 & 9.66 & 1 \\ \cline{2-6}
            & 0.2 & $u_t = 0.449 u_{xx} - 0.852 uu_x$ & 14.77 & 13.99 & 1\\ \cline{2-6}
            & 0.3 & $u_t = 0.43 u_{xx} - 0.77 uu_x$ & 23.01 & 21.51  & 1\\ \hline
    \end{tabular}
\end{table}

Figure \ref{fig:burgers-training} illustrates the HF and LF noisy observation data with sizes ($n_{h1}$, $n_{l1}$), and prediction results of MFGP under noise level $\sigma_\NR=0.2$. The corresponding optimal hyperparameters in MFGP are listed in Table \ref{tab:Burgers-hyperparameters-MFGP}.

\begin{figure}[H] 
    \centering
    \subfloat{\includegraphics[width = 0.34\textwidth]{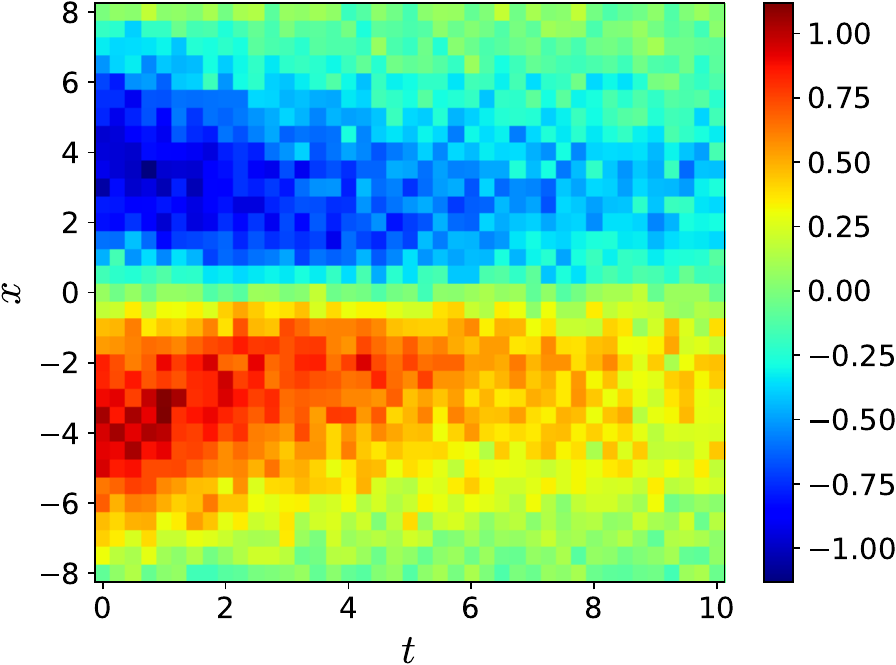}}
    \subfloat{\includegraphics[width = 0.34\textwidth]{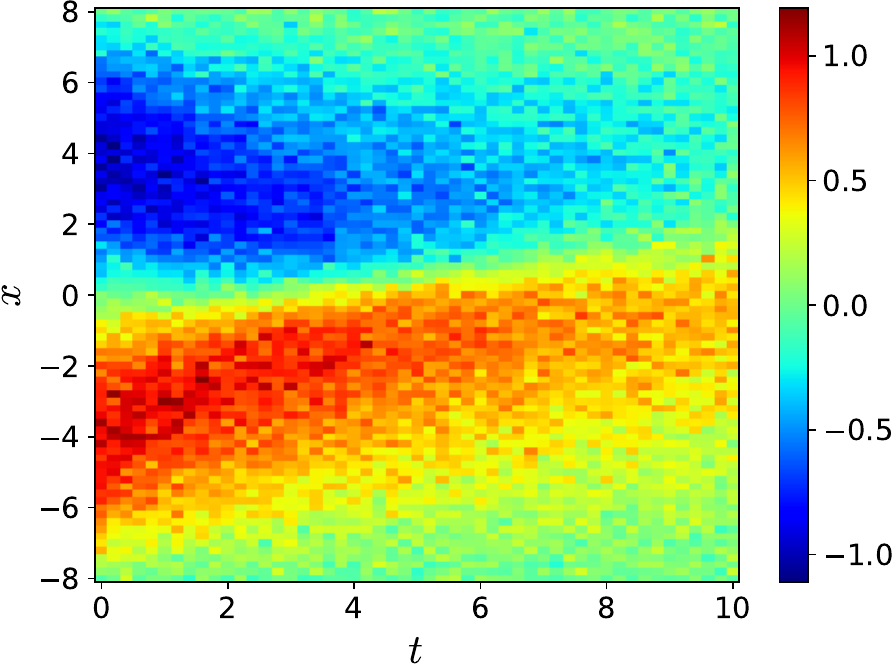}}
    \subfloat{\includegraphics[width = 0.34\textwidth]{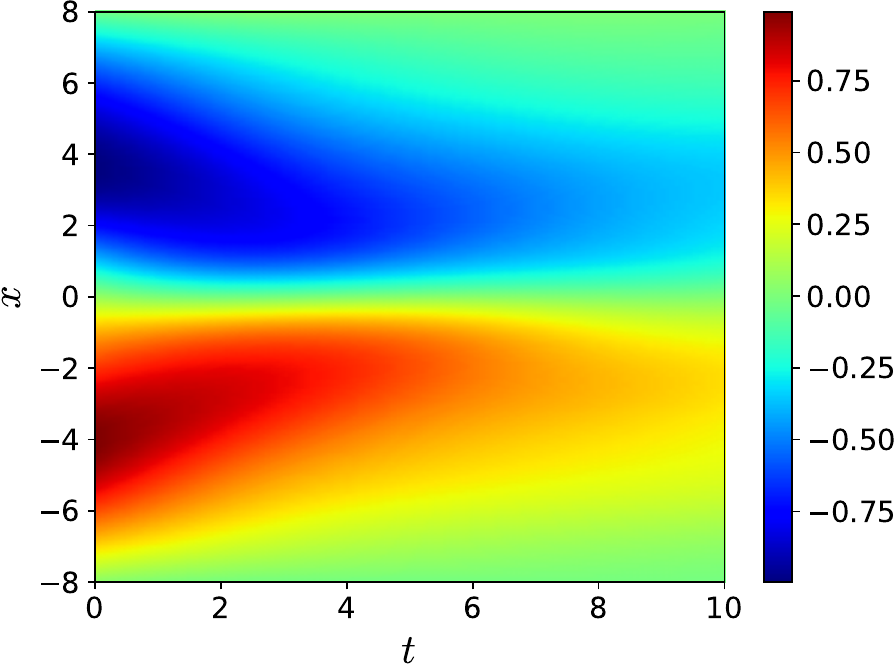}}
    \caption{The HF and LF noisy observation data of sizes ($n_{h1}$, $n_{l1}$), and prediction result of MFGP under $\sigma_\NR=0.2$.}
    \label{fig:burgers-training}
\end{figure}

\begin{table}[H]
    \centering
    \captionsetup{width=0.67\textwidth}
    \caption{Optimal hyperparameters in MFGP, where $\theta^l$ and $(\sigma^l)^2$ are hyperparameters in the GP model level-$l$ $(l=1, 2)$, with the MFGP kernel \(k^2(\theta^2)=k_\rho(\theta_\rho) k_f(\theta_f) + k_\delta(\theta_\delta)\). Here, large $(\theta_{\delta})_1$ implies that the corresponding input (spatial coordinate) has little contribution.}
    \label{tab:Burgers-hyperparameters-MFGP}
    \begin{tabular}{ccccccc}
        \hline
        $(\theta^1)_0$ & $(\theta^1)_1$ & $(\theta^1)_2$ & $(\theta_{\rho})_0$ & $(\theta_{\rho})_1$ & $(\theta_{\rho})_2$ \\
        \hline
        $0.554$ & $14.745$ & $2.580$ & $0.735$ & $65.972$ & $5.607$ \\
        \addlinespace
        \hline
        $(\theta_{f})_0$ & $(\theta_{f})_1$ & $(\theta_{\delta})_0$ & $(\theta_{\delta})_1$ & $(\theta_{\delta})_2$ & $(\sigma_\MF^1)^{2}$ & $(\sigma_\MF^2)^{2}$ \\
        \hline
        $0.746$ & $0.664$ & $0.488$ & $4.581 \times 10^{25}$ & $13.561$ & $0.039$ & $0.037$ \\
    \end{tabular}
\end{table}

\subsubsection{Randomly sampled data}

In the previous part, we only consider the uniform grid data. Given that the MF data size plays an essential role, we focus on testing the performance of MFGP-SINDy using randomly sampled observation data in this part. To explore the impact of HF training data size on the outcome, we fix the size of LF data to $n_{l1}$ $(51 \times 81)$ and randomly select $n_h$ points from the fine spatiotemporal grid as the HF training data. Meanwhile, the results of GP-SINDy are used for comparison.

Figure \ref{fig:burgers-random-error} illustrates $E_\infty$, $E_2$, and $TPR$ under varying HF training data sizes from 100 to 1000 when the noise level $\sigma_\NR \in \{0.1, 0.2, 0.3\}$, where the results are computed on the average over 50 runs. It can be observed that for each noise level, both GP-SINDy and MFGP-SINDy have better outcome with the increasing HF data size $n_h$. Significantly, for noise levels $0.2$ and $0.3$, the $E_\infty$ (the maximum error in the true non-zero coefficients) of GP-SINDy is close to MFGP-SINDy, but $E_2$ and $TPR$ are worse. When the data is abundant, GP-SINDy yields accurate results. The incorporation of MF data in MFGP-SINDy enables effective information fusion, thereby yields satisfactory results with lower computational cost compared with large high resolution SF data.

\begin{figure}[H] 
    \centering
    \subfloat{\includegraphics[width = 0.32\textwidth]{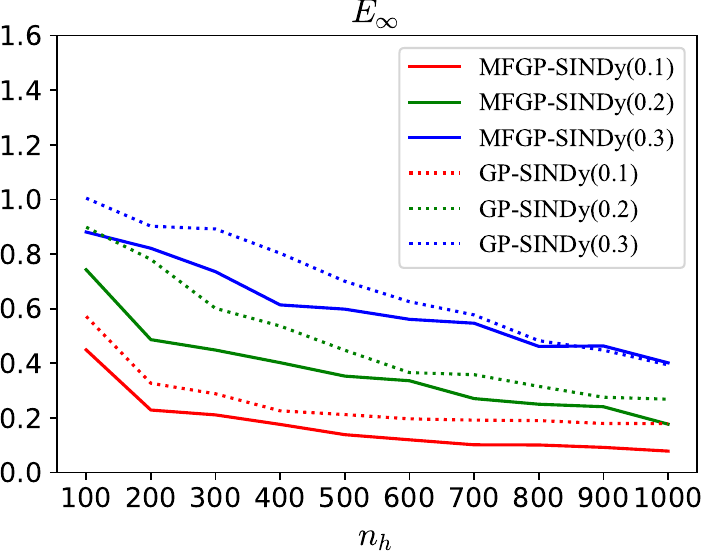}\label{fig:burgers-random-error-1}}
    \hfill
    \subfloat{\includegraphics[width = 0.32\textwidth]{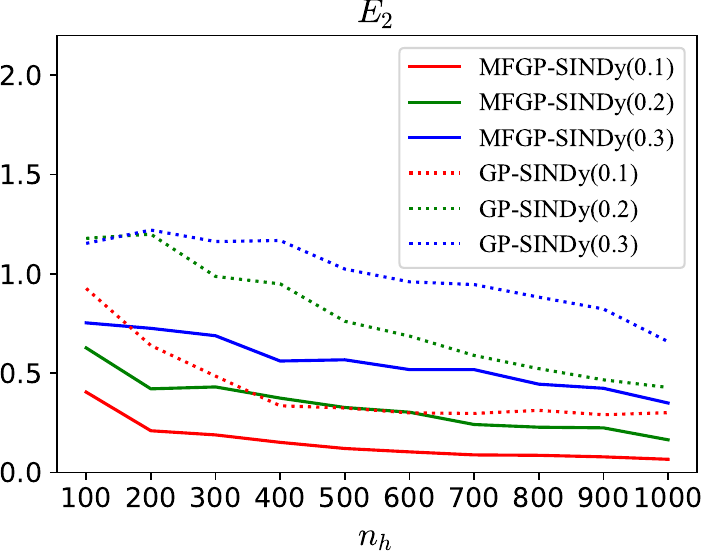}\label{fig:burgers-random-error-2}}
    \hfill
    \subfloat{\includegraphics[width = 0.32\textwidth]{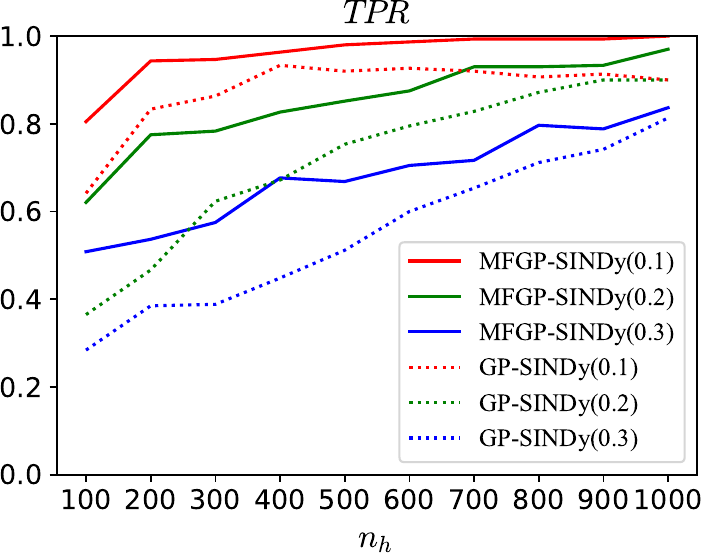}\label{fig:burgers-random-error-3}}
    \caption{The comparison of GP-SINDy and MFGP-SINDy with varying randomly sampled HF data sizes.}
    \label{fig:burgers-random-error}
\end{figure}

\subsubsection{Another multi-fidelity structure}
In this part, we focus on another MF structure, where the only difference with Section \ref{sec:Burgers-MFGP} is the LF model which is obtained by neglecting the nonlinear term in the HF model, such LF model also appears in \cite{chakraborty2021transfer}. The HF and LF models are described by
\begin{equation}\label{equ:burgers-Ano}
    \begin{cases}
        (u_h)_{t} = \nu (u_h)_{xx} - u_h(u_h)_x,\\
        (u_l)_t = \nu (u_l)_{xx}.
    \end{cases}
\end{equation}
Here, the ``clean'' LF data is generated by solving the LF model via the \texttt{spin} class from the \texttt{Chebfun} library \cite{driscoll2014ChebfunGuide} with the same spatiotemporal grid $\Delta {t}_l = 0.25$, $\Delta {x}_l = 0.5$. When we set the training data sizes $n_{h2}$ $(41 \times 33)$, $n_{l2}$ $(41 \times 33)$, the identification results under different $\sigma_\NR$ are listed in Table \ref{tab:Burgers-Ano-symbol}. The LF model in \eqref{equ:burgers-Ano} has different evolution with the HF model, which makes LF data inaccurate, although its effects is not as strong as the coarse grid data in Table \ref{tab:Burgers-MFGP-uniform-symbol}. MFGP-SINDy has the ability to identify the correct function terms using MF data.

\begin{table}[H] 
    \centering
    \caption{Discovered results and error of MFGP-SINDy.}
    \label{tab:Burgers-Ano-symbol}
    \begin{tabular}{c|l|l|l|l|l}
        \hline
        \textbf{MF data size} & $\sigma_\NR$ & \textbf{Discovered equation} & $E_\infty$ (\%) & $E_2 (\%)$ & $TPR$\\
        \hline
        \multirow{4}*{$(n_{h2}, n_{l2})$}
        & 0.02 & $u_t = 0.496 u_{xx} - 0.985 uu_x$ & 1.55 & 1.43 & 1 \\ \cline{2-6}
        & 0.1 & $u_t = 0.479 u_{xx} - 0.92 uu_x$ & 8.03 & 7.42 & 1 \\ \cline{2-6}
        & 0.2 & $u_t = 0.455 u_{xx} - 0.861 uu_x$ & 13.94 & 13.11 & 1 \\ \cline{2-6}
        & 0.3 & $u_t = 0.431 u_{xx} - 0.801 uu_x$ & 19.86 & 18.80 & 1 \\ \hline
    \end{tabular}
\end{table}

\subsection{KdV equation}
\label{sec:KdV}

Our third example is the Korteweg-de Vries (KdV) equation \cite{robertstephany2023pdelearn}, which is given by,
\begin{equation}\label{equ:KdV}
    u_{t} = - u_{xxx} - uu_x,
\end{equation}
with the initial condition $u(x, 0) = \exp(-\pi (x/30)^2)\cos(\pi x/10)$ and the periodic boundary condition $u(-20, t) = u(20, t)$, where $t\in[0, 40]$, $x\in[-20, 20]$. Similar with the Burgers' equation, we generate the HF and LF data by the \texttt{spin} class from the \texttt{Chebfun} library with fine grid $\Delta {t} = 0.002$, $\Delta {x} = 0.015625$ and coarse grid $\Delta {t}_l = 0.1$, $\Delta {x}_l = 1.25$, respectively. The LF data is interpolated linearly on the grid $\widetilde{\Delta {t}_l} = \Delta {t}_l$, $\widetilde{\Delta {x}_l} = \Delta {x}_l/8$ in order to generate more data. Figure \ref{fig:KdV-ground_truth} shows the ground truth (HF data) and LF data before interpolation. The function library contains spatial partial derivatives up to the third-order and polynomials up to the second-order, \ie, $\boldsymbol{\Phi} (u ) = [1, u, u_x, u_{xx}, u_{xxx}, u^2, uu_x, uu_{xx}, uu_{xxx}, u_{x}^2, u_xu_{xx}, u_xu_{xxx}, u_{xx}^2, u_{xx}u_{xxx}, u_{xxx}^2]$. The predictive points are set as the spatiotemporal grid in $ [0, 40] \cup [-20, 20]$ with $\Delta_t' = 0.4$, $\Delta_x' = 0.15625$. The parameters in STWLS are set to $\Lambda = \{i/2, i=0, 1, \cdots, 10\}$, and $\eta = 150$.

\begin{figure}[H] 
    \centering
    \subfloat{\includegraphics[width = 0.34\textwidth]{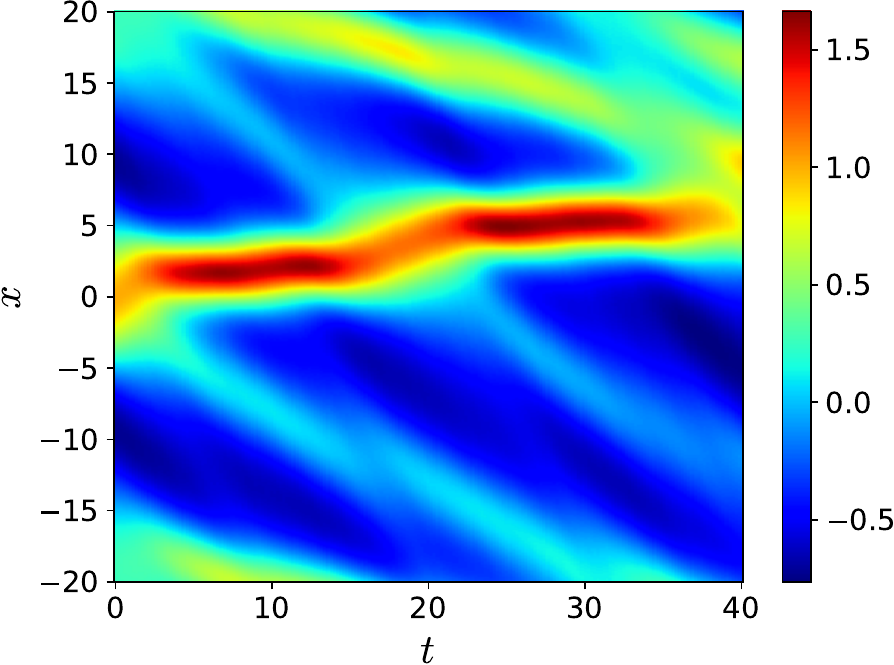}}
    \hspace{0.5cm}
    \subfloat{\includegraphics[width = 0.34\textwidth]{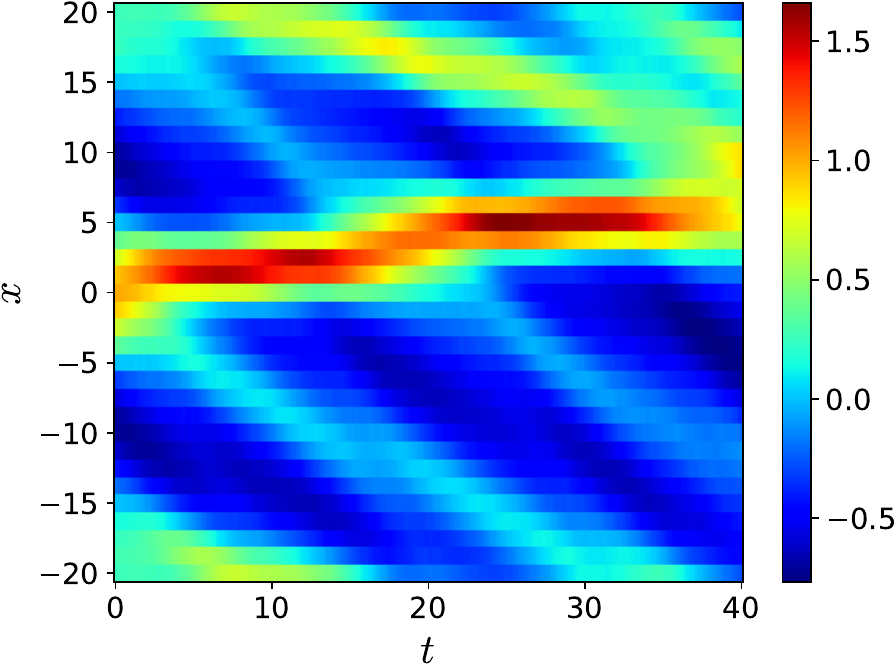}}
    \caption{The HF data and LF data (before interpolation) in KdV equation \eqref{equ:KdV}.}
    \label{fig:KdV-ground_truth}
\end{figure}

The SF training data size is set to $n_{s}$ $(41 \times 81)$ and MF data size is set to $n_{h}$ $(41 \times 81)$, $n_{l}$ $(41 \times 129)$. Here, we test the performance of Tik, rIDENT, SG-Tik, GP-SINDy with SF data size $n_{s}$, and MFGP-SINDy with MF data size $(n_{h}, n_{l})$ under three noise levels $\sigma_\NR \in \{0.03, 0.05, 0.1\}$. Table \ref{tab:KdV-uniform-symbol} shows the symbolic discovered results of GP-SINDy and MFGP-SINDy.
The identification error ($E_2$) of Tik, rIDENT, SG-Tik, GP-SINDy, and MFGP-SINDy are listed in Table \ref{tab:KdV-err-cmp}.
rIDENT has better outcomes than SG-Tik since rIDENT smooths the data repeatedly rather than once in SG-Tik.
Overall, MFGP-SINDy performs best to capture the accurate system. Meanwhile, the LF training data contributes to the final symbolic results, especially when the noise level is relatively high ($\sigma_\NR=0.1$).

\begin{table}[H] 
    \centering
    \captionsetup{width=0.8\textwidth}
    \caption{Identification results of GP-SINDy with SF data size $n_{s}$ and MFGP-SINDy with MF data size $(n_{h}, n_{l})$.}
    \label{tab:KdV-uniform-symbol}
    \begin{tabular}{l|l|l}
        \hline
        $\sigma_\NR$ & \textbf{Discovered equation (GP-SINDy)} & \textbf{Discovered equation (MFGP-SINDy)}  \\
        \hline
         0.03  & $u_t = -0.882 u_{xxx} - 0.875 uu_x$ & $u_t = -0.901 u_{xxx} - 0.894 uu_x$   \\ \hline
         0.05  & $u_t = -0.824 u_{xxx} - 0.813 uu_x$ & $u_t = -0.86 u_{xxx} - 0.853 uu_x$  \\ \hline
         0.1   & $u_t = -0.699 u_{xxx} - 0.685 uu_x$ & $u_t = -0.804 u_{xxx} - 0.804 uu_x$  \\ \hline
    \end{tabular}
\end{table}

\begin{table}[H] 
    \centering
    \captionsetup{width=0.8\textwidth}
    \caption{Identification error $E_2 (\%)$ of Tik, rIDENT, SG-Tik, and GP-SINDy with SF data size $n_{s}$ and MFGP-SINDy with MF data size $(n_{h}, n_{l})$, and we use ``---'' to represent the failed discovery and give the discovered function terms.}
    \label{tab:KdV-err-cmp}
    \begin{tabular}{c|l|l|l|l|l}
        \hline
        $\sigma_\NR$ & Tik (\%) & rIDENT (\%) & SG-Tik (\%) & GP-SINDy (\%) & MFGP-SINDy (\%) \\
        \hline
        0.03 & --- ($uu_x$) & 17.12 & 43.76 & 12.12 & \textbf{10.27} \\ \cline{1-6}
        0.05 & --- ($u_x, uu_x, u_xu_{xx}$) & 28.01 & 48.21 & 18.17 & \textbf{14.36} \\ \cline{1-6}
        0.1 & --- ($u_x, uu_x$) & 43.45 & 54.13 & 30.80 & \textbf{19.63} \\ \hline
    \end{tabular}
\end{table}

\section{Conclusion}
\label{sec:Conclusion}
In this paper, we propose two robust algorithms GP-SINDy and MFGP-SINDy for effective sparse discovery of differential equations. GP-SINDy and MFGP-SINDy are designed for coping with single-fidelity and multi-fidelity observed data, respectively. Both of them are based on Gaussian process regression which eliminate the effect of noise and provide the uncertainty quantification for the inference variables. We compute the variance of time derivatives by their posterior variance in GPR, which is embodied in the weighted least-squares to improve the discovery outcomes. MFGP-SINDy enables to use less amount of high-fidelity data to obtain satisfactory results, which reduces the computational cost.



\section*{Acknowledgments}
We would like to thank Yuchen He and Sung Ha Kang for sharing their code.

\appendix
\section{Computations of the partial derivatives of the kernel functions}
\label{sec:app-der}
\subsection{Partial derivatives of the SE kernel}
\label{sec:app-der-SE}
Before the derivation of the partial derivatives of the MFGP kernel, we first give some useful formulas in this part. For an SE kernel function $k_{\SE }\left(\mathbf{x} ,\mathbf{x}^{\prime } ;\theta \right) =\theta _{0}^{2}\exp\left( -\sum\limits _{s=1}^{D}\frac{\left(\mathbf{x}_{s} -\mathbf{x}_{s}^{\prime }\right)^{2}}{2\theta _{s}^{2}}\right)$, its first-order partial derivative w.r.t. $\mathbf{x}_{j}^{\prime }$ (the $j$-th elements of $\mathbf{x}^{\prime }$) is
\begin{equation}\label{equ:kernel-SE-partial-1}
    \frac{\partial k_{\SE }\left(\mathbf{x} ,\mathbf{x}^{\prime } ;\theta \right)}{\partial \mathbf{x} '_{j}} =\theta _{0}^{2}\exp\left( -\sum\limits _{s=1}^{D}\frac{\left(\mathbf{x}_{s} -\mathbf{x}_{s}^{\prime }\right)^{2}}{2\theta _{s}^{2}}\right)\frac{\mathbf{x}_{j} -\mathbf{x}_{j}^{\prime }}{\theta _{j}^{2}},
\end{equation}
and its first-order partial derivative w.r.t. $\mathbf{x}_{i}$ is given by,
\begin{equation*}\label{equ:kernel-SE-partial-1-add}
    \frac{\partial k_{\SE }\left(\mathbf{x} ,\mathbf{x}^{\prime } ;\theta \right)}{\partial \mathbf{x}_{i}} =-\theta _{0}^{2}\exp\left( -\sum\limits _{s=1}^{D}\frac{\left(\mathbf{x}_{s} -\mathbf{x}_{s}^{\prime }\right)^{2}}{2\theta _{s}^{2}}\right)\frac{\mathbf{x}_{i} -\mathbf{x}_{i}^{\prime }}{\theta _{i}^{2}} =-\frac{\partial k_{\SE }\left(\mathbf{x} ,\mathbf{x}^{\prime } ;\theta \right)}{\partial \mathbf{x} '_{i}}.
\end{equation*}

Its second-order partial derivative w.r.t. $\mathbf{x}_{j}^{\prime }$ and $\mathbf{x}_{i}$ are
\begin{equation}\label{equ:kernel-SE-partial-2}
    \frac{\partial ^{2} k_{\SE }\left(\mathbf{x} ,\mathbf{x}^{\prime } ;\theta \right)}{\partial (\mathbf{x} '_{j})^{2}} =\theta _{0}^{2}\exp\left( -\sum\limits _{s=1}^{D}\frac{\left(\mathbf{x}_{s} -\mathbf{x}_{s}^{\prime }\right)^{2}}{2\theta _{s}^{2}}\right)\frac{1}{\theta _{j}^{2}}\left(\frac{\left(\mathbf{x}_{j} -\mathbf{x}_{j}^{\prime }\right)^{2}}{\theta _{j}^{2}} -1\right),
\end{equation}
and
\begin{equation*}\label{equ:kernel-SE-partial-2-add1}
    \frac{\partial ^{2} k_{\SE }\left(\mathbf{x} ,\mathbf{x}^{\prime } ;\theta \right)}{\partial (\mathbf{x}_{i})^{2}} =\frac{\partial ^{2} k_{\SE }\left(\mathbf{x} ,\mathbf{x}^{\prime } ;\theta \right)}{\partial (\mathbf{x} '_{i})^{2}}.
\end{equation*}
Meanwhile, its second-order cross partial derivative
is
\begin{equation*}\label{equ:kernel-SE-partial-2-add2}
    \frac{\partial ^{2} k_{\SE }\left(\mathbf{x} ,\mathbf{x}^{\prime } ;\theta \right)}{\partial \mathbf{x}_{i} \partial \mathbf{x} '_{j}} =
    \begin{cases}
        -\theta _{0}^{2}\exp\left( -\sum\limits _{s=1}^{D}\frac{\left(\mathbf{x}_{s} -\mathbf{x}_{s}^{\prime }\right)^{2}}{2\theta _{s}^{2}}\right)\frac{\mathbf{x}_{i} -\mathbf{x}_{i}^{\prime }}{\theta _{i}^{2}}\frac{\mathbf{x}_{j} -\mathbf{x}_{j}^{\prime }}{\theta _{j}^{2}} , & i\neq j,\\
        \theta _{0}^{2}\exp\left( -\sum\limits _{s=1}^{D}\frac{\left(\mathbf{x}_{s} -\mathbf{x}_{s}^{\prime }\right)^{2}}{2\theta _{s}^{2}}\right)\frac{1}{\theta _{j}^{2}}\left( 1-\frac{\left(\mathbf{x}_{j} -\mathbf{x}_{j}^{\prime }\right)^{2}}{\theta _{j}^{2}}\right) =-\frac{\partial ^{2} k_{\SE }\left(\mathbf{x} ,\mathbf{x}^{\prime } ;\theta \right)}{\partial (\mathbf{x} '_{j})^{2}} , & i=j.
    \end{cases}
\end{equation*}

\subsection{Partial derivatives of the MFGP kernel}
\label{sec:app-der-MFGP}
The MFGP kernel is given by,
\begin{equation*}\label{equ:app-mfgp-kernel}
    k^{l}\left(\left(\mathbf{x} ,\overline{f^{l-1}} (\mathbf{x} )\right) ,\left(\mathbf{x} ',\overline{f^{l-1}} (\mathbf{x} ')\right) ;\theta ^{l}\right) =k_{\rho }(\mathbf{x} ,\mathbf{x} ';\theta _{\rho }) k_{f}\left(\overline{f^{l-1}} (\mathbf{x} ),\overline{f^{l-1}} (\mathbf{x} ');\theta _{f}\right) +k_{\delta }(\mathbf{x} ,\mathbf{x} ';\theta _{\delta }).
\end{equation*}
For simplicity, we rewrite it as
\[
k((\mathbf{x} ,f(\mathbf{x} )),(\mathbf{x} ',f(\mathbf{x} ')))=k_{\rho } (\mathbf{x} ,\mathbf{x} ')k_{f} (f (\mathbf{x} ),f (\mathbf{x} '))+k_{\delta } (\mathbf{x} ,\mathbf{x} '),
\]
where $k_{\rho } (\mathbf{x} ,\mathbf{x} ')$, $k_{f} (f (\mathbf{x} ),f (\mathbf{x} '))$ and $k_{\delta } (\mathbf{x} ,\mathbf{x} ')$ are three SE kernel functions, and its first-order partial derivative w.r.t. $\mathbf{x}_{j}^{\prime }$ is
\begin{equation*}\label{equ:kernel-MFGP-partial-1}
    \begin{aligned}
        & \frac{\partial }{\partial \mathbf{x} '_{j}} k((\mathbf{x} ,f(\mathbf{x} )),(\mathbf{x} ',f(\mathbf{x} ')))\\
       = & \frac{\partial k_{\rho } (\mathbf{x} ,\mathbf{x} ')}{\partial \mathbf{x} '_{j}} k_{f} (f (\mathbf{x} ),f (\mathbf{x} '))+k_{\rho } (\mathbf{x} ,\mathbf{x} ')\frac{k_{f} (f (\mathbf{x} ),f (\mathbf{x} '))}{\partial \mathbf{x} '_{j}} +\frac{\partial k_{\delta } (\mathbf{x} ,\mathbf{x} ')}{\partial \mathbf{x} '_{j}}\\
       = & \frac{\partial k_{\rho } (\mathbf{x} ,\mathbf{x} ')}{\partial \mathbf{x} '_{j}} k_{f} (f (\mathbf{x} ),f (\mathbf{x} '))+k_{\rho } (\mathbf{x} ,\mathbf{x} ')\frac{\partial k_{f} (f (\mathbf{x} ),f (\mathbf{x} '))}{\partial f (\mathbf{x} ')}\frac{\partial f (\mathbf{x} ')}{\partial \mathbf{x} '_{j}} +\frac{\partial k_{\delta } (\mathbf{x} ,\mathbf{x} ')}{\partial \mathbf{x} '_{j}} ,
    \end{aligned}
\end{equation*}
where $\frac{\partial k_{\rho } (\mathbf{x} ,\mathbf{x} ')}{\partial \mathbf{x} '_{j}}$, $\frac{\partial k_{f} (f (\mathbf{x} ),f (\mathbf{x} '))}{\partial f (\mathbf{x} ')}$, and $\frac{\partial k_{\delta } (\mathbf{x} ,\mathbf{x} ')}{\partial \mathbf{x} '_{j}}$ are computed by \eqref{equ:kernel-SE-partial-1}. Here, $\frac{\partial f (\mathbf{x} ')}{\partial \mathbf{x} '_{j}}$ is the prediction of the first-order partial derivatives in the LF model.

The second-order partial derivative of MFGP kernel is given by,
\begin{equation*}\label{equ:kernel-MFGP-partial-2-1}
    \begin{aligned}
        & \frac{\partial ^{2}}{\partial \mathbf{x}_{i} \partial \mathbf{x} '_{j}} k((\mathbf{x} ,f(\mathbf{x} )),(\mathbf{x} ',f(\mathbf{x} ')))\\
       = & \frac{\partial }{\partial \mathbf{x}_{i}}\left[\frac{\partial k_{\rho } (\mathbf{x} ,\mathbf{x} ')}{\partial \mathbf{x} '_{j}} k_{f} (f (\mathbf{x} ),f (\mathbf{x} '))+k_{\rho } (\mathbf{x} ,\mathbf{x} ')\frac{\partial k_{f} (f (\mathbf{x} ),f (\mathbf{x} '))}{\partial f (\mathbf{x} ')}\frac{\partial f (\mathbf{x} ')}{\partial \mathbf{x} '_{j}} +\frac{\partial k_{\delta } (\mathbf{x} ,\mathbf{x} ')}{\partial \mathbf{x} '_{j}}\right]\\
       = & \frac{\partial ^{2} k_{\rho } (\mathbf{x} ,\mathbf{x} ')}{\partial \mathbf{x}_{i} \partial \mathbf{x} '_{j}} k_{f} (f (\mathbf{x} ),f (\mathbf{x} '))+\frac{\partial k_{\rho } (\mathbf{x} ,\mathbf{x} ')}{\partial \mathbf{x} '_{j}}\frac{\partial k_{f} (f (\mathbf{x} ),f (\mathbf{x} '))}{\partial f(\mathbf{x} )}\frac{\partial f (\mathbf{x} )}{\partial \mathbf{x}_{i}}\\
        & +\frac{\partial k_{\rho } (\mathbf{x} ,\mathbf{x} ')}{\partial \mathbf{x}_{i}}\frac{\partial k_{f} (f (\mathbf{x} ),f (\mathbf{x} '))}{\partial f (\mathbf{x} ')}\frac{\partial f (\mathbf{x} ')}{\partial \mathbf{x} '_{j}} +k_{\rho } (\mathbf{x} ,\mathbf{x} ')\frac{\partial }{\partial \mathbf{x}_{i}}\left(\frac{\partial k_{f} (f (\mathbf{x} ),f (\mathbf{x} '))}{\partial f (\mathbf{x} ')}\right)\frac{\partial f (\mathbf{x} ')}{\partial \mathbf{x} '_{j}} +\frac{\partial ^{2} k_{\delta } (\mathbf{x} ,\mathbf{x} ')}{\partial \mathbf{x}_{i} \partial \mathbf{x} '_{j}}\\
       = &  \frac{\partial ^{2} k_{\rho } (\mathbf{x} ,\mathbf{x} ')}{\partial \mathbf{x}_{i} \partial \mathbf{x} '_{j}} k_{f} (f (\mathbf{x} ),f (\mathbf{x} '))-\frac{\partial k_{\rho } (\mathbf{x} ,\mathbf{x} ')}{\partial \mathbf{x} '_{j}}\frac{\partial k_{f} (f (\mathbf{x} ),f (\mathbf{x} '))}{\partial f(\mathbf{x} ')}\frac{\partial f (\mathbf{x} )}{\partial \mathbf{x}_{i}}\\
        & -\frac{\partial k_{\rho } (\mathbf{x} ,\mathbf{x} ')}{\partial \mathbf{x} '_{i}}\frac{\partial k_{f} (f (\mathbf{x} ),f (\mathbf{x} '))}{\partial f (\mathbf{x} ')}\frac{\partial f (\mathbf{x} ')}{\partial \mathbf{x} '_{j}} +k_{\rho } (\mathbf{x} ,\mathbf{x} ')\frac{\partial ^{2} k_{f} (f (\mathbf{x} ),f (\mathbf{x} '))}{\partial f (\mathbf{x} )\partial f (\mathbf{x} ')}\frac{\partial f (\mathbf{x} )}{\partial \mathbf{x}_{i}}\frac{\partial f (\mathbf{x} ')}{\partial \mathbf{x} '_{j}} +\frac{\partial ^{2} k_{\delta } (\mathbf{x} ,\mathbf{x} ')}{\partial \mathbf{x}_{i} \partial \mathbf{x} '_{j}}\\
       = & \frac{\partial ^{2} k_{\rho } (\mathbf{x} ,\mathbf{x} ')}{\partial \mathbf{x}_{i} \partial \mathbf{x} '_{j}} k_{f} (f (\mathbf{x} ),f (\mathbf{x} '))-\frac{\partial k_{f} (f (\mathbf{x} ),f (\mathbf{x} '))}{\partial f(\mathbf{x} ')}\left(\frac{\partial k_{\rho } (\mathbf{x} ,\mathbf{x} ')}{\partial \mathbf{x} '_{j}}\frac{\partial f (\mathbf{x} )}{\partial \mathbf{x}_{i}} +\frac{\partial k_{\rho } (\mathbf{x} ,\mathbf{x} ')}{\partial \mathbf{x} '_{i}}\frac{\partial f (\mathbf{x} ')}{\partial \mathbf{x} '_{j}}\right)\\
        & +k_{\rho } (\mathbf{x} ,\mathbf{x} ')\frac{\partial ^{2} k_{f} (f (\mathbf{x} ),f (\mathbf{x} '))}{\partial f (\mathbf{x} )\partial f (\mathbf{x} ')}\frac{\partial f (\mathbf{x} )}{\partial \mathbf{x}_{i}}\frac{\partial f (\mathbf{x} ')}{\partial \mathbf{x} '_{j}} +\frac{\partial ^{2} k_{\delta } (\mathbf{x} ,\mathbf{x} ')}{\partial \mathbf{x}_{i} \partial \mathbf{x} '_{j}} ,
    \end{aligned}
\end{equation*}
where $\frac{\partial ^{2} k_{\rho } (\mathbf{x} ,\mathbf{x} ')}{\partial \mathbf{x}_{i} \partial \mathbf{x} '_{j}} $, $\frac{\partial ^{2} k_{f} (f (\mathbf{x} ),f (\mathbf{x} '))}{\partial ( f (\mathbf{x} ))^{2}} $, $\frac{\partial ^{2} k_{\delta } (\mathbf{x} ,\mathbf{x} ')}{\partial \mathbf{x}_{i} \partial \mathbf{x} '_{j}}$ are computed through \eqref{equ:kernel-SE-partial-2}, $\frac{\partial f (\mathbf{x} )}{\partial \mathbf{x}_{i}}$ and $\frac{\partial f (\mathbf{x} ')}{\partial \mathbf{x} '_{j}}$ are the first-order partial derivatives of the LF model.

Another second-order partial derivative of the MFGP kernel is given by,
\begin{equation*}\label{equ:kernel-MFGP-partial-2-2}
    \begin{aligned}
        & \frac{\partial ^{2}}{\partial (\mathbf{x} '_{j})^{2}} k((\mathbf{x} ,f(\mathbf{x} )),(\mathbf{x} ',f(\mathbf{x} ')))\\
       = & \frac{\partial }{\partial \mathbf{x} '_{j}}\left[\frac{\partial k_{\rho } (\mathbf{x} ,\mathbf{x} ')}{\partial \mathbf{x} '_{j}} k_{f} (f(\mathbf{x} ),f(\mathbf{x} '))+k_{\rho } (\mathbf{x} ,\mathbf{x} ')\frac{\partial k_{f} (f(\mathbf{x} ),f(\mathbf{x} '))}{\partial f(\mathbf{x} ')}\frac{\partial f(\mathbf{x} ')}{\partial \mathbf{x} '_{j}} +\frac{\partial k_{\delta } (\mathbf{x} ,\mathbf{x} ')}{\partial \mathbf{x} '_{j}}\right]\\
       = & \frac{\partial ^{2} k_{\rho } (\mathbf{x} ,\mathbf{x} ')}{\partial (\mathbf{x} '_{j})^{2}} k_{f} (f(\mathbf{x} ),f(\mathbf{x} '))+2\frac{\partial k_{\rho } (\mathbf{x} ,\mathbf{x} ')}{\partial \mathbf{x} '_{j}}\frac{\partial k_{f} (f(\mathbf{x} ),f(\mathbf{x} '))}{\partial f(\mathbf{x} ')}\frac{\partial f(\mathbf{x} ')}{\partial \mathbf{x} '_{j}}\\
        & +k_{\rho } (\mathbf{x} ,\mathbf{x} ')\frac{\partial }{\partial \mathbf{x} '_{j}}\left[\frac{\partial k_{f} (f(\mathbf{x} ),f(\mathbf{x} '))}{\partial f(\mathbf{x} ')}\frac{\partial f(\mathbf{x} ')}{\partial \mathbf{x} '_{j}}\right] +\frac{\partial ^{2} k_{\delta } (\mathbf{x} ,\mathbf{x} ')}{\partial (\mathbf{x} '_{j})^{2}}.
    \end{aligned}
\end{equation*}
Here,
\[
\begin{aligned}
    & \frac{\partial }{\partial \mathbf{x} '_{j}}\left[\frac{\partial k_{f} (f(\mathbf{x} ),f(\mathbf{x} '))}{\partial f(\mathbf{x} ')}\frac{\partial f(\mathbf{x} ')}{\partial \mathbf{x} '_{j}}\right]\\
   = & \frac{\partial ^{2} k_{f} (f(\mathbf{x} ),f(\mathbf{x} '))}{\partial f(\mathbf{x} ')\partial \mathbf{x} '_{j}}\frac{\partial f(\mathbf{x} ')}{\partial \mathbf{x} '_{j}} +\frac{\partial k_{f} (f(\mathbf{x} ),f(\mathbf{x} '))}{\partial f(\mathbf{x} ')}\frac{\partial ^{2} f(\mathbf{x} ')}{\partial (\mathbf{x} '_{j})^{2}}\\
   = & \frac{\partial ^{2} k_{f} (f(\mathbf{x} ),f(\mathbf{x} '))}{\partial ( f(\mathbf{x} '))^{2}}\left(\frac{\partial f(\mathbf{x} ')}{\partial \mathbf{x} '_{j}}\right)^{2} +\frac{\partial k_{f} (f(\mathbf{x} ),f(\mathbf{x} '))}{\partial f(\mathbf{x} ')}\frac{\partial ^{2} f(\mathbf{x} ')}{\partial (\mathbf{x} '_{j})^{2}} ,
\end{aligned}
\]
where $\frac{\partial ^{2} f_{*1} (\mathbf{x} ')}{\partial (\mathbf{x} '_{j})^{2}}$ is the second-order partial derivative in the LF model.

Analogously, the third-order partial derivative of the MFGP kernel is,
\begin{equation*}\label{equ:kernel-MFGP-partial-3}
    \begin{aligned}
        & \frac{\partial ^{3}}{\partial (\mathbf{x}_{j} ')^{3}} k((\mathbf{x} ,f(\mathbf{x} )),(\mathbf{x} ',f(\mathbf{x} ')))=\frac{\partial }{\partial \mathbf{x}_{j} '}\left(\frac{\partial ^{2}}{\partial (\mathbf{x} '_{j})^{2}} k((\mathbf{x} ,f(\mathbf{x} )),(\mathbf{x} ',f(\mathbf{x} ')))\right)\\
       = & \frac{\partial ^{3} k_{\rho } (\mathbf{x} ,\mathbf{x} ')}{\partial (\mathbf{x} '_{j})^{3}} k_{f} (f(\mathbf{x} ),f(\mathbf{x} '))+3\frac{\partial ^{2} k_{\rho } (\mathbf{x} ,\mathbf{x} ')}{\partial (\mathbf{x} '_{j})^{2}}\frac{\partial k_{f} (f(\mathbf{x} ),f(\mathbf{x} '))}{\partial f(\mathbf{x} ')}\frac{\partial f(\mathbf{x} ')}{\partial \mathbf{x} '_{j}}\\
        & +2\frac{\partial k_{\rho } (\mathbf{x} ,\mathbf{x} ')}{\partial \mathbf{x} '_{j}}\frac{\partial }{\partial \mathbf{x} '_{j}}\left(\frac{\partial k_{f} (f(\mathbf{x} ),f(\mathbf{x} '))}{\partial f(\mathbf{x} ')}\frac{\partial f(\mathbf{x} ')}{\partial \mathbf{x} '_{j}}\right)\\
        & +\frac{\partial k_{\rho } (\mathbf{x} ,\mathbf{x} ')}{\partial \mathbf{x} '_{j}}\left(\frac{\partial ^{2} k_{f} (f(\mathbf{x} ),f(\mathbf{x} '))}{\partial ( f(\mathbf{x} '))^{2}}\left(\frac{\partial f(\mathbf{x} ')}{\partial \mathbf{x} '_{j}}\right)^{2} +\frac{\partial k_{f} (f(\mathbf{x} ),f(\mathbf{x} '))}{\partial f(\mathbf{x} ')}\frac{\partial ^{2} f(\mathbf{x} ')}{\partial (\mathbf{x} '_{j})^{2}}\right)\\
        & +k_{\rho } (\mathbf{x} ,\mathbf{x} ')\frac{\partial }{\partial \mathbf{x} '_{j}}\left(\frac{\partial ^{2} k_{f} (f(\mathbf{x} ),f(\mathbf{x} '))}{\partial ( f(\mathbf{x} '))^{2}}\left(\frac{\partial f(\mathbf{x} ')}{\partial \mathbf{x} '_{j}}\right)^{2} +\frac{\partial k_{f} (f(\mathbf{x} ),f(\mathbf{x} '))}{\partial f(\mathbf{x} ')}\frac{\partial ^{2} f(\mathbf{x} ')}{\partial (\mathbf{x} '_{j})^{2}}\right) +\frac{\partial ^{3} k_{\delta } (\mathbf{x} ,\mathbf{x} ')}{\partial (\mathbf{x} '_{j})^{3}}\\
       = & \frac{\partial ^{3} k_{\rho } (\mathbf{x} ,\mathbf{x} ')}{\partial (\mathbf{x} '_{j})^{3}} k_{f} (f(\mathbf{x} ),f(\mathbf{x} '))+3\frac{\partial ^{2} k_{\rho } (\mathbf{x} ,\mathbf{x} ')}{\partial (\mathbf{x} '_{j})^{2}}\frac{\partial k_{f} (f(\mathbf{x} ),f(\mathbf{x} '))}{\partial f(\mathbf{x} ')}\frac{\partial f(\mathbf{x} ')}{\partial \mathbf{x} '_{j}}\\
        & +3\frac{\partial k_{\rho } (\mathbf{x} ,\mathbf{x} ')}{\partial \mathbf{x} '_{j}}\left(\frac{\partial ^{2} k_{f} (f(\mathbf{x} ),f(\mathbf{x} '))}{\partial ( f(\mathbf{x} '))^{2}}\left(\frac{\partial f(\mathbf{x} ')}{\partial \mathbf{x} '_{j}}\right)^{2} +\frac{\partial k_{f} (f(\mathbf{x} ),f(\mathbf{x} '))}{\partial f(\mathbf{x} ')}\frac{\partial ^{2} f(\mathbf{x} ')}{\partial (\mathbf{x} '_{j})^{2}}\right)\\
        & +k_{\rho } (\mathbf{x} ,\mathbf{x} ' )\frac{\partial }{\partial \mathbf{x} '_{j}}\left(\frac{\partial ^{2} k_{f} (f(\mathbf{x} ),f(\mathbf{x} '))}{\partial ( f(\mathbf{x} '))^{2}}\left(\frac{\partial f(\mathbf{x} ')}{\partial \mathbf{x} '_{j}}\right)^{2} +\frac{\partial k_{f} (f(\mathbf{x} ),f(\mathbf{x} '))}{\partial f(\mathbf{x} ')}\frac{\partial ^{2} f(\mathbf{x} ')}{\partial (\mathbf{x} '_{j})^{2}}\right) +\frac{\partial ^{3} k_{\delta } (\mathbf{x} ,\mathbf{x} ')}{\partial (\mathbf{x} '_{j})^{3}}.
    \end{aligned}
\end{equation*}
Here
\[
\begin{aligned}
    & \frac{\partial }{\partial \mathbf{x} '_{j}}\left(\frac{\partial ^{2} k_{f} (f(\mathbf{x} ),f(\mathbf{x} '))}{\partial ( f(\mathbf{x} '))^{2}}\left(\frac{\partial f(\mathbf{x} ')}{\partial \mathbf{x} '_{j}}\right)^{2} +\frac{\partial k_{f} (f(\mathbf{x} ),f(\mathbf{x} '))}{\partial f(\mathbf{x} ')}\frac{\partial ^{2} f(\mathbf{x} ')}{\partial (\mathbf{x} '_{j})^{2}}\right)\\
   = & \frac{\partial ^{3} k_{f} (f(\mathbf{x} ),f(\mathbf{x} '))}{\partial ( f(\mathbf{x} '))^{3}}\left(\frac{\partial f(\mathbf{x} ')}{\partial \mathbf{x} '_{j}}\right)^{3} +2\frac{\partial ^{2} k_{f} (f(\mathbf{x} ),f(\mathbf{x} '))}{\partial ( f(\mathbf{x} '))^{2}}\frac{\partial f(\mathbf{x} ')}{\partial \mathbf{x} '_{j}}\frac{\partial ^{2} f(\mathbf{x} ')}{\partial (\mathbf{x} '_{j})^{2}}\\
    & +\frac{\partial ^{2} k_{f} (f(\mathbf{x} ),f(\mathbf{x} '))}{\partial ( f(\mathbf{x} '))^{2}}\frac{\partial f(\mathbf{x} ')}{\partial \mathbf{x} '_{j}}\frac{\partial ^{2} f(\mathbf{x} ')}{\partial (\mathbf{x} '_{j})^{2}} +\frac{\partial k_{f} (f(\mathbf{x} ),f(\mathbf{x} '))}{\partial f(\mathbf{x} ')}\frac{\partial ^{3} f(\mathbf{x} ')}{\partial (\mathbf{x} '_{j})^{3}} .
\end{aligned}
\]

\bibliographystyle{elsarticle-num-names}
\bibliography{reference}
\end{document}